\documentclass[10pt,reqno]{amsart}
\usepackage{amsmath,amssymb,latexsym,esint,cite,mathrsfs}
\usepackage{verbatim,wasysym,cite,relsize,color,soul}
\usepackage[latin1]{inputenc}
\usepackage{microtype}
\usepackage{color,enumitem,graphicx}
\usepackage[colorlinks=true,urlcolor=blue, citecolor=red,linkcolor=blue,
linktocpage,pdfpagelabels, bookmarksnumbered,bookmarksopen]{hyperref}
\usepackage[hyperpageref]{backref}
\usepackage[english]{babel}
\usepackage{tikz}
\usepackage[font=small,skip=5pt]{caption}
\usepackage{amsrefs}

\usepackage{geometry}
\numberwithin{equation}{section}
\newtheorem{theorem}{Theorem}[section]
\newtheorem{lemma}[theorem]{Lemma}

\newtheorem{proposition}[theorem]{Proposition}
\newtheorem{corollary}[theorem]{Corollary}

\newtheoremstyle{remarkstyle}
{}{}{\itshape }{ }{\bfseries}{.}{ }{\thmname{#1}\thmnumber{ #2}\thmnote{ (#3)}}
\theoremstyle{remarkstyle}
\newtheorem{remark}{Remark}[section]
\newtheorem{definition}{Definition}[section]

\newcommand{\R}{\mathbb R}
\newcommand{\C}{\mathbb C}

\newcommand{\eps}{\epsilon}
\DeclareMathOperator*{\loc}{loc}

\DeclareMathOperator*{\inls}{(INLS)}

\DeclareMathOperator*{\gamc}{{\gamma_c}}

\DeclareMathOperator*{\ima}{Im}
\DeclareMathOperator*{\rea}{Re}
\newcommand{\scal}[1]{\left\langle #1 \right\rangle}

\title[Scattering in a weighted space for the defocusing INLS]{Scattering theory in weighted $L^2$ space for a class of the defocusing inhomogeneous nonlinear Schr\"odinger equation} 

\author[V. D. Dinh]{Van Duong Dinh}
\address[V. D. Dinh]{Laboratoire Paul Painlev\'e UMR 8524, Universit\'e de Lille CNRS, 59655 Villeneuve d'Ascq Cedex, France
	and 
	Department of Mathematics, HCMC University of Education, 280 An Duong Vuong, Ho Chi Minh, Vietnam}
\email{contact@duongdinh.com}


\begin{document}

\maketitle

\begin{abstract}
In this paper, we consider the following inhomogeneous nonlinear Schr\"odinger equation (INLS)
\[
i\partial_t u + \Delta u + \mu |x|^{-b} |u|^\alpha u = 0, \quad (t,x)\in \R \times \R^d
\]
with $b, \alpha>0$. First, we revisit the local well-posedness in $H^1(\R^d)$ for $\inls$ of Guzm\'an [Nonlinear Anal. Real World Appl. 37 (2017), 249-286] and give an improvement of this result in the two and three spatial dimensional cases. Second, we study the decay of global solutions for the defocusing $\inls$, i.e. $\mu=-1$ when $0<\alpha<\alpha^\star$ where $\alpha^\star = \frac{4-2b}{d-2}$ for $d\geq 3$, and $\alpha^\star = \infty$ for $d=1, 2$ by assuming that the initial data belongs to the weighted $L^2$ space $\Sigma =\{u \in H^1(\R^d) : |x| u \in L^2(\R^d) \}$. Finally, we combine the local theory and the decaying property to show the scattering in $\Sigma$ for the defocusing $\inls$ in the case $\alpha_\star<\alpha<\alpha^\star$, where $\alpha_\star = \frac{4-2b}{d}$.
\end{abstract}


\section{Introduction}
\setcounter{equation}{0}
\label{S1}

One of the most important equations in nonlinear optics is the nonlinear Schr\"odinger equation (NLS). It models the propagation of intense laser beams in a homogeneous bulk medium with a Kerr nonlinearity. It is well-known that NLS governed the beam propagation in a homogeneous bulk media cannot support stable high-power propagation. It was suggested at the end of the last century that stable high-power propagation can be obtained in plasma by sending a preliminary laser beam that creates a channel with a reduced electron density, and thus reduces the nonlinear
inside the channel (see e.g. \cite{Grill, LT}). In this physical model, the beam propagation can be described by the inhomogeneous nonlinear Schr\"odinger equation of the form
\begin{align} \label{INLS-V}
i\partial_t u + \Delta u + V(x) |u|^\alpha u = 0, \quad (t,x)\in \R \times \R^d,
\end{align}
where $u$ is the electric field in laser and optics, $\alpha > 0$ is the power of nonlinear interaction, and the potential $V(x)$ is proportional to the electron density. In \cite{TM}, Towers and Malomed observed by means of variational approximation and direct simulations that for a certain type of nonlinear medium, \eqref{INLS-V} gives rise to completely stable beams.

When the potential $V$ is constant, \eqref{INLS-V} becomes the standard nonlinear Schr\"odinger equation which has been studied extensively in the past decades (see e.g. \cite{Cazenave, Tao}).

When the potential $V$ is a non-constant bounded function, Merle \cite{Merle} showed the existence and nonexistence of minimal blow-up solutions to \eqref{INLS-V} with $\alpha=\frac{4}{d}$ and $V_1\leq V(x) \leq V_2$, where $V_1$ and $V_2$ are positive constants. Later, Rapha\"el and Szeftel \cite{RS} extended the work of Merle \cite{Merle} and established sufficient conditions for the existence, uniqueness, and charaterization of minimial blow-up solutions to the equation. Fibich and Wang \cite{FW} and Liu and Wang \cite{LWW} investigated the stability and instability of solitary waves for \eqref{INLS-V} with $\alpha \geq \frac{4}{d}$ and $V(x) = V(\epsilon x)$, where $\epsilon>0$ is a small parameter and $V \in C^4(\R^d) \cap L^\infty(\R^d)$. 

When the potential $V$ is unbounded, the problem becomes more involved. The case $V(x)= |x|^b, b>0$ was studied in several works. Chen and Guo \cite{CG} and Chen \cite{Chen} proved sharp criteria for the global existence and blow-up. Zhu \cite{Zhu} studied the existence and dynamical properties of blow-up solutions. When $V$ behaves like $|x|^{-b}$ with $b>0$, Bouard and Fukuizumi \cite{BF} studied the stability of standing waves for \eqref{INLS-V} with $\alpha<\frac{4-2b}{d}$. Fukuizumi and Ohta \cite{FO} established the instability of standing waves for \eqref{INLS-V} with $\alpha>\frac{4-2b}{d}$. 

In this paper, we consider the following type of inhomogeneous nonlinear Schr\"odinger equations
\[
\left\{
\begin{array}{ccl}
i\partial_t u + \Delta u + \mu |x|^{-b} |u|^\alpha u &=& 0, \quad (t,x)\in \R \times \R^d,\\
\left.u\right|_{t=0}&=& u_0,
\end{array} 
\right. \tag{INLS}
\]
where $u: \R \times \R^d \rightarrow \C, u_0:\R^d \rightarrow \C$, $\mu = \pm 1$, $\alpha>0$, and $b>0$. The terms $\mu=1$ and $\mu=-1$ correspond to the focusing and defocusing cases respectively. This equation plays an important role as a limiting equation in the analysis of (1.1) with $V(x) \sim |x|^{-b}$ as $|x| \rightarrow \infty$ (see e.g. \cite{Genoud, Genoud-Thesis, GS}).

Before reviewing known results for $\inls$, we recall some facts for this equation.  First, we note that $\inls$ is invariant under the scaling
\[
u_\lambda(t,x):= \lambda^{\frac{2-b}{\alpha}} u(\lambda^2 t, \lambda x), \quad \lambda>0.
\]
An easy computation shows
\[
\|u_\lambda(0)\|_{\dot{H}^\gamma(\R^d)} = \lambda^{\gamma+\frac{2-b}{\alpha}-\frac{d}{2}} \|u_0\|_{\dot{H}^\gamma(\R^d)}.
\]
Thus, the critical Sobolev exponent is given by 
\begin{align}
\gamc := \frac{d}{2}-\frac{2-b}{\alpha}. \label{critical exponent}
\end{align}
Moreover, $\inls$ has the following conserved quantities:
\begin{align}
M(u(t))&:= \|u(t)\|^2_{L^2(\R^d)} = M(u_0), \label{mass conservation} \\
E(u(t)) &:=  \frac{1}{2}\|\nabla u(t)\|^2_{L^2(\R^d)} - \mu G(t) = E(u_0), \label{energy conservation} 
\end{align}
where 
\begin{align}
G(t):=\frac{1}{\alpha+2} \int_{\R^d} |x|^{-b}|u(t,x)|^{\alpha+2} dx. \label{define G}
\end{align}

The well-posedness for $\inls$  in $H^1(\R^d)$ was firstly studied by Genoud and Stuart in \cite[Appendix]{GS} (see also \cite{Genoud-Thesis}). The proof is based on the abstract theory developed by Cazenave \cite{Cazenave} which does not use Strichartz estimates. More precisely, the authors showed that the focusing $\inls$ with $0<b<\min\{2,d\}$ is well posed in $H^1(\R^d)$: 
\begin{itemize}
\item locally if $0<\alpha<\alpha^\star$,
\item globally for any initial data if $0<\alpha <\alpha_\star$,
\item globally for small initial data if $\alpha_\star \leq \alpha <\alpha^\star$.
\end{itemize}
Here $\alpha_\star$ and $\alpha^\star$ are defined by
\begin{align}
\renewcommand*{\arraystretch}{1.2}
\alpha_\star:=\frac{4-2b}{d}, \quad \alpha^\star := \left\{ \begin{array}{c l}
\frac{4-2b}{d-2} & \text{if } d\geq 3, \\
\infty &\text{if } d=1, 2.
\end{array} \right. \label{alpha exponents}
\end{align}

In the case $\alpha=\alpha_\star$ ($L^2$-critical), Genoud in \cite{Genoud} showed that the focusing $\inls$ with $0<b<\min\{2,d\}$ is globally well-posed in $H^1(\R^d)$ assuming $u_0 \in H^1(\R^d)$ and
\[
\|u_0\|_{L^2(\R^d)} <\|Q\|_{L^2(\R^d)},
\]
where $Q$ is the unique nonnegative, radially symmetric, decreasing solution of the ground state equation
\begin{align}
\Delta Q -Q +|x|^{-b}|Q|^{\frac{4-2b}{d}} Q=0. \label{ground state equation 1}
\end{align}
Also, Combet and Genoud in \cite{CG-JEE} established the classification of minimal mass blow-up solutions for the focusing $L^2$-critical $\inls$. 

In the case $\alpha_\star <\alpha<\alpha^\star$, Farah in \cite{Farah} showed that the focusing $\inls$ with $0<b<\min \{2,d\}$ is globally well-posedness in $H^1(\R^d)$ assuming $u_0 \in H^1(\R^d)$ and
\begin{align}
E(u_0)^{\gamc} M(u_0)^{1-\gamc} &< E(Q)^{\gamc} M(Q)^{1-\gamc}, \label{condition 1}\\
\|\nabla u_0\|^{\gamc}_{L^2(\R^d)} \|u_0\|_{L^2(\R^d)}^{1-\gamc} &< \|\nabla Q\|^{\gamc}_{L^2(\R^d)} \|Q\|^{1-\gamc}_{L^2(\R^d)}, \label{condition 2}
\end{align}
where $Q$ is the unique nonnegative, radially symmetric, decreasing solution of the ground state equation
\begin{align}
\Delta Q- Q + |x|^{-b} |Q|^\alpha Q =0. \label{ground state equation 2}
\end{align}
Afterwards, Farah and Guzm\'an in \cite{FG-JDE, FG-BBMS} proved that the above global solution is scattering under the radial condition of the initial data. In \cite{Farah}, Farah also proved that if $u_0 \in \Sigma$ satisfies $(\ref{condition 1})$ and 
\begin{align}
\|\nabla u_0\|^{\gamc}_{L^2(\R^d)} \|u_0\|_{L^2(\R^d)}^{1-\gamc} &> \|\nabla Q\|^{\gamc}_{L^2(\R^d)} \|Q\|^{1-\gamc}_{L^2(\R^d)},
\label{condition 3}
\end{align}
then the finite time blow-up in $H^1(\R^d)$ must occur. This result was later extended to radial data by the author in \cite{Dinh}. Note that the existence and uniqueness of nonnegative, radially symmetric, decreasing solutions to $(\ref{ground state equation 1})$ and $(\ref{ground state equation 2})$  were proved by Toland \cite{Toland} and Yanagida \cite{Yanagida} (see also Genoud and Stuart \cite{GS}). Their results hold under the assumption $0<b<\min\{2, d\}$ and $0<\alpha <\alpha^\star$. 

Recently, Guzm\'an in \cite{Guzman} used Strichartz estimates and the contraction mapping argument to establish the well-posedness for $\inls$ in Sobolev spaces. More precisely, he showed (among other things) that:
\begin{itemize}
\item if $0<\alpha<\alpha_\star$ and $0<b<\min\{2, d\}$, then $\inls$ is locally well-posed in $L^2(\R^d)$. Thus, it is globally well-posed in $L^2(\R^d)$ by mass conservation.
\item if $0<\alpha<\widetilde{\alpha}, 0<b<\widetilde{b}$ and $\max\{0, \gamc\}<\gamma \leq \min\left\{\frac{d}{2},1 \right\}$ where 
\begin{align}
\renewcommand*{\arraystretch}{1.2}
\widetilde{\alpha}:= \left\{ \begin{array}{c l}
\frac{4-2b}{d-2\gamma} & \text{if } \gamma<\frac{d}{2}, \\
\infty &\text{if } \gamma=\frac{d}{2},
\end{array} \right. \quad \text{and} \quad  \widetilde{b}:= \left\{ \begin{array}{c l}
\frac{d}{3} & \text{if } d=1, 2, 3, \\
2 &\text{if } d\geq 4,
\end{array} \right. \label{define tilde alpha and b}
\end{align}
then $\inls$ is locally well-posedness in $H^\gamma(\R^d)$. 
\item if $\alpha_\star<\alpha<\widetilde{\alpha}$, $0<b<\widetilde{b}$ and $\gamc<\gamma\leq \min\left\{\frac{d}{2}, 1\right\}$, then $\inls$ is globally well-posed in $H^\gamma(\R^d)$ for small initial data. 
\end{itemize}
In particular, we have the following local well-posedness in the energy space for $\inls$.
\begin{theorem}[\cite{Guzman}] \label{theorem Guzman result}
Let $d\geq 2, 0<b<\widetilde{b}$ and $0<\alpha<\alpha^\star$, where
\[
\renewcommand*{\arraystretch}{1.2}
\widetilde{b}:=\left\{
\begin{array}{cl}
\frac{d}{3} &\text{if } d=2,3,\\
2 &\text{if } d\geq 4.
\end{array}
\right. 
\]
Then $\inls$ is locally well-posed in $H^1(\R^d)$. Moreover, local solutions to $\inls$ satisfy $u \in L^p_{\loc}$ $((-T_*,T^*),W^{1,q} (\R^d))$ for any Schr\"odinger admissible pair $(p,q)$, where $(-T_*,T^*)$ is the maximal time interval of existence.
\end{theorem} 
Note that the result of Guzm\'an \cite{Guzman} about the local well-posedness for $\inls$ in $H^1(\R^d)$ is weaker than the one of Genoud and Stuart \cite{GS}. More precisely, it does not treat the case $d=1$, and there is a restriction on the validity of $b$ when $d=2$ or $3$. Although the result showed by Genoud and Stuart is strong, but one does not know whether local solutions to $\inls$ belong to $L^p_{\text{loc}}((-T_*,T^*), W^{1,q}(\R^d))$ for any Schr\"odinger admissible pair $(p,q)$. This property plays an important role in proving the scattering for the defocusing $\inls$. Our first result is the following local well-posedness in $H^1(\R^d)$ which improves Guzm\'an's result on the range of $b$ in the two and three spatial dimensions. 
\begin{theorem} \label{theorem local existence}
Let 
\[
d\geq 4, \quad 0<b<2, \quad 0 <\alpha <\alpha^\star,
\]
or 
\[
d=3, \quad 0< b <1, \quad 0<\alpha <\alpha^\star,
\]
or 
\[
d=3, \quad 1 \leq b<\frac{3}{2}, \quad 0<\alpha <\frac{6-4b}{2b-1},
\]
or
\[
d=2, \quad 0<b<1, \quad 0 <\alpha< \alpha^\star.
\]
Then $\inls$ is locally well-posed in $H^1(\R^d)$. Moreover, local solutions to $\inls$ satisfy $u \in L^p_{\loc}$ $((-T_*,T^*),W^{1,q} (\R^d))$ for any Schr\"odinger admissible pair $(p,q)$, where $(-T_*,T^*)$ is the maximal time interval of existence.
\end{theorem} 
We will see in Section $\ref{S3}$ that one can not expect a similar result as in Theorem $\ref{theorem Guzman result}$ and Theorem $\ref{theorem local existence}$ holds in the one dimensional case by using Strichartz estimates. Thus the local well-posedness in the energy space for $\inls$ of Genoud and Stuart is the best known result. 

\begin{remark}
	The methods used to show the local well-posedness in $H^1(\R^d)$ in this paper and in \cite{Guzman} are not applicable to treat the critical regularity. After the submission of this paper, the author learns that there are recent papers \cite{LS, KLS} addressing the local well-posedness for $\inls$ with critical regularities. The proofs of these results are based on weighted Strichartz and Sobolev estimates.  	
\end{remark}

The local well-posedness \footnote{The local well-posedness in $H^1(\R^d)$ of Genoud and Stuart is still valid for the defocusing case.} of Genoud and Stuart in \cite{GS, Genoud-Thesis} combines with the conservations of mass and energy immediately give the global well-posedness in $H^1(\R^d)$ for the defocusing $\inls$, i.e. $\mu=-1$. To our knowledge, there are few results concerning long-time dynamics of the defocusing $\inls$. Let us introduce the following weighted space
\[
\Sigma :=H^1(\R^d)\cap L^2(\R^d,|x|^2dx) = \{ u \in H^1(\R^d) : |x| u \in L^2(\R^d) \},
\]
equipped with the norm
\[
\|u\|_{\Sigma}:= \|u\|_{H^1(\R^d)} + \|x u\|_{L^2(\R^d)}.
\]
Our next result concerns with the decay of global solutions to the defocusing $\inls$ by assuming the initial data in $\Sigma$. 

\begin{theorem} \label{theorem decay property}
Let $0<b<\min\{2, d\}$. Let $u_0 \in \Sigma$ and $u \in C(\R, H^1(\R^d))$ be the unique global solution to the defocusing $\inls$. Then the following properties hold:
\begin{itemize}
\item[1.] If $\alpha \in [\alpha_\star, \alpha^\star)$, then for every 
\begin{align}
\renewcommand*{\arraystretch}{1.2}
\left\{
\begin{array}{l l}
2\leq q \leq \frac{2d}{d-2} & \text{if } d\geq 3, \\
2 \leq q <\infty & \text{if } d=2, \\
2 \leq q \leq \infty & \text{if } d=1,
\end{array}
\right. \label{define q}
\end{align}
there exists $C>0$ such that 
\begin{align}
\|u(t)\|_{L^q(\R^d)} \leq C|t|^{-d\left(\frac{1}{2}-\frac{1}{q}\right)},  \label{decay property 1}
\end{align}
for all $t\in \R\backslash \{0\}$.
\item[2.] If $\alpha \in (0, \alpha_\star)$, then for every $q$ given in $(\ref{define q})$, there exists $C>0$ such that
\begin{align}
\|u(t)\|_{L^q(\R^d)} \leq C|t|^{-\frac{d(2b+d\alpha)}{4}\left(\frac{1}{2}-\frac{1}{q}\right)},  \label{decay property 2}
\end{align}
for all $t\in \R\backslash \{0\}$.
\end{itemize}
\end{theorem}
This result extends the well-known result of the classical (i.e. $b=0$) nonlinear Schr\"odinger equation (see e.g. \cite[Theorem 7.3.1]{Cazenave} and references cited therein). 

We then use this decay and Strichartz estimates to show the scattering for global solutions to the defocusing $\inls$. Due to the singularity of $|x|^{-b}$, the scattering result does not cover the same range of exponents $b$ and $\alpha$ as in Theorem $\ref{theorem local existence}$. More precisely, we have the following:
\begin{theorem} \label{theorem scattering weighted space}
Let 
\[
d\geq 4, \quad 0<b<2, \quad \alpha_\star \leq \alpha <\alpha^\star,
\]
or 
\[
d=3, \quad 0< b <1, \quad \frac{5-2b}{3}<\alpha <3-2b,
\]
or
\[
d=2, \quad 0<b<1, \quad \alpha_\star \leq \alpha< \alpha^\star.
\]
Let $u_0 \in \Sigma$ and $u$ be the unique global solution to the defocusing $\inls$. Then there exist unique $u_0^\pm \in \Sigma$ such that
\[
\lim_{t\rightarrow \pm \infty} \|e^{-it\Delta}u(t)- u_0^\pm\|_{\Sigma} =0.
\]
\end{theorem}
In this theorem, we only consider the case $\alpha \in [\alpha_\star, \alpha^\star)$. A similar result in the case $\alpha \in (0,\alpha^\star)$ is possible, but it is complicated due to the rate of decays in $(\ref{decay property 2})$. We will give some comments about this case in the end of Section $\ref{S6}$. 

The proof of Theorem \ref{theorem scattering weighted space} is based on a standard argument (see e.g. \cite{Cazenave}) using decay estimates of global solutions given in Theorem \ref{theorem decay property} and nonlinear estimates given in Lemmas \ref{lem scattering weighted space d geq 4}, \ref{lem scattering weighted space d=3}, \ref{lem scattering weighted space d=2}. Due to the appearance of the singular term $|x|^{-b}$, we need more care in showing nonlinear estimates. We refer the reader to Section \ref{S6} for more details.

\begin{remark}
	After this paper was submitted to arXiv, there are several works studying the scattering in the energy space for $\inls$, for instance, \cite{Dinh-JEE}, \cite{Campos}, \cite{Dinh-arXiv}, and \cite{XZ}.
\end{remark}

This paper is organized as follows. In the next section, we introduce some notation and recall Strichartz estimates for the linear Schr\"odinger equation. In Section 3, we prove the local well-posedness given in Theorem $\ref{theorem local existence}$. In Section 4, we derive the virial identity and show the pseudo-conformal conservation law related to the defocusing $\inls$. We will give the proof of Theorem $\ref{theorem decay property}$ in Section 5. Finally, Section 6 is devoted to the scattering result of Theorem $\ref{theorem scattering weighted space}$. 

\section{Preliminaries} 
\setcounter{equation}{0}
\label{S2}

In the sequel, the notation $A \lesssim B$ denotes an estimate of the form $A\leq CB$ for some constant $C>0$. The constant $C>0$ may change from line to line. 
\subsection{Nonlinearity} \label{subsection nonlinearity}
Let $F(x, z):=|x|^{-b} f(z)$ with $b>0$ and $f(z):=|z|^\alpha z$. The complex derivatives of $f$ are
\[
\partial_zf(z) = \frac{\alpha+2}{2}|z|^\alpha, \quad \partial_{\overline{z}} f(z) = \frac{\alpha}{2} |z|^{\alpha-2} z^2.
\] 
We have for $z, w \in \C$, 
\[
f(z) - f(w) = \int_0^1 \Big(\partial_z f(w+\theta(z-w)) (z-w) + \partial_{\overline{z}} f(w+\theta(z-w)) (\overline{z}-\overline{w}) \Big) d\theta.
\]
Thus,
\begin{align}
|F(x, z)-F(x, w)| \lesssim |x|^{-b} (|z|^\alpha+|w|^\alpha) |z-w|. \label{nonlinear inequality}
\end{align}
To deal with the singularity $|x|^{-b}$, we have the following remark.
\begin{remark}[\cite{Guzman}] \label{rem singularity bound}
Let $B=B(0,1) = \{x\in \R^d: |x|<1\}$ and $B^c=\R^d \backslash B$. Then
\[
\||x|^{-b}\|_{L^\gamma_x(B)} <\infty, \quad \text{if} \quad \frac{d}{\gamma}>b,
\]
and
\[
\||x|^{-b}\|_{L^\gamma_x(B^c)} <\infty, \quad \text{if} \quad \frac{d}{\gamma}<b.
\]
\end{remark}

\subsection{Strichartz estimates} \label{subsection strichartz estimate}
Let $I \subset \R$ and $p, q \in [1,\infty]$. We define the mixed norm
\[
\|u\|_{L^p_t(I, L^q_x)} := \Big( \int_I \Big( \int_{\R^d} |u(t,x)|^q dx \Big)^{\frac{p}{q}} dt\Big)^{\frac{1}{p}}
\] 
with a usual modification when either $p$ or $q$ are infinity. When there is no risk of confusion, we may write $L^p_t L^q_x$ instead of $L^p_t(I,L^q_x)$. We also use $L^p_{t,x}$ when $p=q$.
\begin{definition}
A pair $(p,q)$ is said to be \textbf{Schr\"odinger admissible}, for short $(p,q) \in S$, if 
\[
(p,q) \in [2,\infty]^2, \quad (p,q,d) \ne (2,\infty,2), \quad \frac{2}{p}+\frac{d}{q} = \frac{d}{2}.
\]
\end{definition}
We denote for any spacetime slab $I\times \R^d$,
\begin{align}
\|u\|_{S(L^2, I)}:= \sup_{(p,q) \in S} \|u\|_{L^p_t(I,L^q_x)}, \quad \|v\|_{S'(L^2,I)}:=\inf_{(p,q)\in S} \|v\|_{L^{p'}_t(I, L^{q'}_x)}. \label{strichartz norm}
\end{align}
We next recall well-known Strichartz estimates for the linear Schr\"odinger equation. We refer the reader to \cite{Cazenave, Tao} for more details.
\begin{proposition} \label{prop strichartz}
Let $u$ be a solution to the linear Schr\"odinger equation, namely
\[
u(t)= e^{it\Delta}u_0 + \int_0^t e^{i(t-s)\Delta} F(s) ds,
\]
for some data $u_0, F$. Then we have
\begin{align}
\|u\|_{S(L^2,\R)} \lesssim \|u_0\|_{L^2_x} + \|F\|_{S'(L^2, \R)}. \label{strichartz estimate}
\end{align}
\end{proposition}

\section{Local existence} 
\setcounter{equation}{0}
\label{S3}
In this section, we give the proof of the local well-posedness given in Theorem $\ref{theorem local existence}$. To prove this result, we need the following lemmas which give some estimates of the nonlinearity.
\begin{lemma} [\cite{Guzman}] \label{lem-non-est-d-geq3}
Let $d\geq 4$ and $0<b<2$ or $d=3$ and $0<b<1$. Let $0<\alpha <\alpha^\star$ and $I = [0,T]$. Then there exist $\theta_1, \theta_2>0$ such that
\begin{align}
\||x|^{-b} |u|^\alpha v\|_{S'(L^2,I)} &\lesssim \left(T^{\theta_1} + T^{\theta_2}\right) \|\nabla u\|^\alpha_{S(L^2,I)} \|v\|_{S(L^2, I)}, \label{non-est-1-d-geq3} \\
\|\nabla(|x|^{-b} |u|^\alpha u)\|_{S'(L^2,I)} & \lesssim \left(T^{\theta_1} + T^{\theta_2}\right) \|\nabla u\|^{\alpha+1}_{S(L^2, I)}. \label{non-est-2-d-geq3}
\end{align}
\end{lemma}
The proof of this result is given in \cite[Lemma 3.4]{Guzman}. For reader's convenience and later use, we give some details. \newline
\noindent \textit{Proof of Lemma $\ref{lem-non-est-d-geq3}$.}
We bound
\begin{align*}
\||x|^{-b} |u|^\alpha v\|_{S'(L^2,I)} &\leq \||x|^{-b} |u|^\alpha v\|_{S'(L^2(B),I)} + \||x|^{-b} |u|^\alpha v\|_{S'(L^2(B^c),I)} =: A_1 + A_2, \\
\|\nabla(|x|^{-b} |u|^\alpha u)\|_{S'(L^2,I)} & \leq \|\nabla(|x|^{-b} |u|^\alpha u)\|_{S'(L^2(B),I)} + \|\nabla(|x|^{-b} |u|^\alpha u)\|_{S'(L^2(B^c),I)} =:B_1 +B_2. 
\end{align*}
\indent \underline{\bf On $B$.} By H\"older inequality and Remark $\ref{rem singularity bound}$,
\begin{align*}
A_1 \leq \||x|^{-b} |u|^\alpha v\|_{L^{p_1'}_t(I, L^{q_1'}_x(B))} &\lesssim \||x|^{-b}\|_{L^{\gamma_1}_x(B)} \||u|^\alpha v\|_{L^{p_1'}_t(I, L^{\upsilon_1}_x)} \\
& \lesssim \|u\|^\alpha_{L^{m_1}_t(I, L^{n_1}_x)} \|v\|_{L^{p_1}_t(I, L^{q_1}_x)} \\
&\lesssim  T^{\theta_1} \|\nabla u\|^\alpha_{L^{p_1}_t(I, L^{q_1}_x)} \|v\|_{L^{p_1}_t(I, L^{q_1}_x)},
\end{align*}
provided that $(p_1, q_1)\in S$ and 
\[
\frac{1}{q_1'} = \frac{1}{\gamma_1} +\frac{1}{\upsilon_1}, \quad \frac{d}{\gamma_1} >b, \quad \frac{1}{\upsilon_1} = \frac{\alpha}{n_1} +\frac{1}{q_1}, \quad \frac{1}{p_1'} = \frac{\alpha}{m_1} +\frac{1}{p_1}, \quad \theta_1 =\frac{\alpha}{m_1} -\frac{\alpha}{p_1},
\]
and 
\[
q_1 < d, \quad \frac{1}{n_1} = \frac{1}{q_1}-\frac{1}{d}.
\]
Here the last condition ensures the Sobolev embedding $\dot{W}^{1, q_1}(\R^d) \subset L^{n_1}(\R^d)$. We see that condition $\frac{d}{\gamma_1}>b$ implies
\begin{align}
\frac{d}{\gamma_1} = d-\frac{d(\alpha+2)}{q_1}+\alpha>b \quad \text{or}\quad q_1>\frac{d(\alpha+2)}{d+\alpha-b}. \label{condition q_1 estimate nonlinear d geq 3}
\end{align}
Let us choose \[
q_1=\frac{d(\alpha+2)}{d+\alpha-b} +\eps,
\] 
for some $0<\eps\ll 1$ to be chosen later. By taking $\eps>0$ small enough, we see that $q_1<d$ implies $d>b+2$ which is true since we are considering $d\geq 4, 0<b<2$ or $d=3, 0<b<1$. On the other hand, using $0<\alpha<\alpha^\star$ and choosing $\eps>0$ sufficiently small, we see that $2<q_1<\frac{2d}{d-2}$. It remains to check $\theta_1>0$. This condition is equivalent to
\[
\frac{\alpha}{m_1}-\frac{\alpha}{p_1}=1-\frac{\alpha+2}{p_1}>0 \quad \text{or} \quad p_1>\alpha+2.
\] 
Since $(p_1, q_1)\in S$, the above inequality implies
\[
\frac{d}{2}-\frac{d}{q_1}=\frac{2}{p_1} <\frac{2}{\alpha+2}.
\]
A direct computation shows
\[
d(\alpha+2)[ 4-2b -(d-2)\alpha] + \eps (d+\alpha-b)(4-d(\alpha+2))>0
\]
Since $\alpha \in (0, \alpha^\star)$, we see that $4-2b-(d-2)\alpha>0$. Thus, by taking $\eps>0$ sufficiently small, the above inequality holds true. Therefore, we have for a sufficiently small value of $\eps$,
\begin{align}
A_1 \lesssim T^{\theta_1}\|\nabla u\|^\alpha_{S(L^2,I)} \|u\|_{S(L^2, I	)}. \label{non-est-1-d-geq3-prof1}
\end{align}
We next bound
\[
B_1 \leq \||x|^{-b} \nabla(|u|^\alpha u)\|_{S'(L^2(B),I)} + \||x|^{-b-1} |u|^\alpha u\|_{S'(L^2(B),I)}=:B_{11}+B_{12}.
\]
The term $B_{11}$ is treated similarly as for $A_1$ by using the fractional chain rule. We obtain
\begin{align}
B_{11} \lesssim T^{\theta_1} \|\nabla u\|^{\alpha+1}_{S(L^2, I)}, \label{non-est-2-d-geq3-prof1}
\end{align}
provided $\eps>0$ is taken small enough. Using Remark $\ref{rem singularity bound}$, we estimate
\begin{align*}
B_{12}\leq \||x|^{-b-1}|u|^\alpha u\|_{L^{p_1'}_t(I, L^{q_1'}_x(B))} &\lesssim \||x|^{-b-1}\|_{L^{\gamma_1}_x(B)} \||u|^\alpha u\|_{L^{p_1'}_t(I, L^{\upsilon_1}_x)}\\
&\lesssim \|u\|^\alpha_{L^{m_1}_t(I, L^{n_1}_x)} \|u\|_{L^{p_1}_t(I, L^{n_1}_x)} \\
&\lesssim T^{\theta_1}\|\nabla u\|^{\alpha+1}_{L^{p_1}_t(I, L^{q_1}_x)}, 
\end{align*}
provided that $(p_1, q_1)\in S$ and 
\[
\frac{1}{q_1'}=\frac{1}{\gamma_1} + \frac{1}{\upsilon_1}, \quad \frac{d}{\gamma_1}>b+1, \quad \frac{1}{\upsilon_1}=\frac{\alpha+1}{n_1}, \quad \frac{1}{p_1'}=\frac{\alpha}{m_1}+\frac{1}{p_1}, \quad \theta_1=\frac{\alpha}{m_1}-\frac{\alpha}{p_1},
\]
and
\[
q_1 <d, \quad \frac{1}{n_1}=\frac{1}{q_1}-\frac{1}{d}. 
\]
We see that
\[
\frac{d}{\gamma_1}= d-\frac{d(\alpha+2)}{q_1}+\alpha+1>b+1 \quad \text{or}\quad q_1>\frac{d(\alpha+2)}{d+\alpha-b}.
\]
The last condition is similar to $(\ref{condition q_1 estimate nonlinear d geq 3})$. Thus, by choosing $q_1$ as above, we obtain for $\eps>0$ small enough,
\begin{align}
B_{12} \lesssim T^{\theta_1}\|\nabla u\|^{\alpha+1}_{S(L^2,I)}. \label{non-est-2-d-geq3-prof2}
\end{align}
\indent \underline{\bf On $B^c$.} Let us choose the following Schr\"odinger admissible pair
\[
p_2=\frac{4(\alpha+2)}{(d-2)\alpha}, \quad q_2=\frac{d(\alpha+2)}{d+\alpha}.
\]
Let $m_2, n_2$ be such that
\begin{align}
\frac{1}{q_2'}=\frac{\alpha}{n_2} + \frac{1}{q_2}, \quad \frac{1}{p_2'}=\frac{\alpha}{m_2} + \frac{1}{p_2}. \label{choose m_2 n_2 estimate nonlinearity d geq 3}
\end{align}
A direct computation shows
\[
\theta_2:=\frac{\alpha}{m_2}-\frac{\alpha}{p_2}=1-\frac{\alpha+2}{p_2}=1-\frac{(d-2)\alpha}{4}>0.
\]
Note that in our consideration, we always have $(d-2)\alpha <4$. Moreover, it is easy to check that
\[
\frac{1}{n_2}=\frac{1}{q_2}-\frac{1}{d}.
\]
It allows us to use the Sobolev embedding $\dot{W}^{1, q_2}(\R^d) \subset L^{n_2}(\R^d)$. By H\"older inequality with $(\ref{choose m_2 n_2 estimate nonlinearity d geq 3})$,
\begin{align}
A_2\leq \||x|^{-b}|u|^\alpha v\|_{L^{p_2'}_t(I, L^{q_2'}_x(B^c))} &\lesssim \||x|^{-b}\|_{L^\infty_x(B^c)} \||u|^\alpha v\|_{L^{p_2'}_t(I, L^{q_2'}_x)} \nonumber \\
&\lesssim \|u\|^\alpha_{L^{m_2}_t(I, L^{n_2}_x)}\|v\|_{L^{p_2}_t(I, L^{q_2}_x)} \nonumber \\
&\lesssim T^{\theta_2} \|\nabla u\|^\alpha_{L^{p_2}_t(I, L^{q_2}_x)} \|v\|_{L^{p_2}_t(I, L^{q_2}_x)}. \nonumber
\end{align}
We thus get
\[
A_2 \lesssim T^{\theta_2} \|\nabla u\|^\alpha_{S(L^2, I)} \|v\|_{S(L^2, I)}.\label{non-est-1-d-geq3-prof2}
\]
We now bound
\[
B_2 \leq \||x|^{-b} \nabla(|u|^\alpha u)\|_{S'(L^2(B^c),I)} + \||x|^{-b-1} |u|^\alpha u\|_{S'(L^2(B^c),I)}=:B_{21}+B_{22}.
\]
The term $B_{21}$ is treated similarly by using the fractional chain rule, and we obtain
\begin{align}
B_{21} \lesssim T^{\theta_2} \|\nabla u\|^{\alpha+1}_{S(L^2,I)}.\label{non-est-2-d-geq3-prof3}
\end{align}
Finally, we estimate
\begin{align}
B_{22}\leq \||x|^{-b-1}|u|^\alpha u\|_{L^{p_2'}_t(I, L^{q_2'}_x(B^c))} &\lesssim \||x|^{-b-1}\|_{L^{d}_x(B^c)} \|u\|^\alpha_{L^{m_2}_t(I, L^{n_2}_x)} \|u\|_{L^{p_2}_t(I, L^{n_2}_x)} \nonumber \\
&\lesssim T^{\theta_2}\|\nabla u\|^{\alpha+1}_{L^{p_2}_t(I, L^{q_2}_x)}. \nonumber
\end{align}
Note that $\frac{1}{q_2'}=\frac{\alpha+1}{n_2}+\frac{1}{d}$. This shows that
\[
B_{22}\lesssim T^{\theta_2}\|\nabla u\|^{\alpha+1}_{S(L^2, I)}. \label{non-est-2-d-geq3-prof4}
\]
Combining \eqref{non-est-1-d-geq3-prof1}--\eqref{non-est-2-d-geq3-prof4}, we complete the proof.
\hfill $\Box$

In the three dimensional case, we also have the following extension. 
\begin{lemma} \label{lem-non-est-d3}
Let $d=3$. Let $1\leq b<\frac{3}{2}$ and $0<\alpha<\frac{6-4b}{2b-1}$ and $I=[0,T]$. Then there exists $\theta_1, \theta_2>0$ such that
\begin{align}
\||x|^{-b} |u|^\alpha v\|_{S'(L^2,I)} &\lesssim \left(T^{\theta_1} + T^{\theta_2}\right) \|\scal{\nabla} u\|^\alpha_{S(L^2,I)} \|v\|_{S(L^2, I)}, \label{non-est-1-d3} \\
\|\nabla(|x|^{-b} |u|^\alpha u)\|_{S'(L^2,I)} & \lesssim \left(T^{\theta_1} + T^{\theta_2}\right) \|\scal{\nabla} u\|^{\alpha+1}_{S(L^2, I)}. \label{non-est-2-d3}
\end{align}
\end{lemma}

\begin{proof}
We use the notations $A_1, A_2, B_{11}, B_{12}, B_{21}$ and $B_{22}$ introduced in the proof of Lemma $\ref{lem-non-est-d-geq3}$. \newline
\indent \underline{\bf On $B$.} By H\"older inequality and Remark $\ref{rem singularity bound}$,
\begin{align*}
A_1 \leq \||x|^{-b} |u|^\alpha v\|_{L^{p_1'}_t(I, L^{q_1'}_x(B))} &\lesssim \||x|^{-b}\|_{L^{\gamma_1}_x(B)} \||u|^\alpha v\|_{L^{p_1'}_t(I, L^{\upsilon_1}_x)} \\
& \lesssim \|u\|^\alpha_{L^{m_1}_t(I, L^{n_1}_x)} \|v\|_{L^{p_1}_t(I, L^{q_1}_x)} \\
&\lesssim  T^{\theta_1} \|\scal{\nabla} u\|_{L^{p_1}_t(I, L^{q_1}_x)} \|v\|_{L^{p_1}_t(I, L^{q_1}_x)},
\end{align*}
provided that $(p_1, q_1)\in S$ and 
\[
\frac{1}{q_1'} = \frac{1}{\gamma_1} +\frac{1}{\upsilon_1}, \quad \frac{3}{\gamma_1} >b, \quad \frac{1}{\upsilon_1} = \frac{\alpha}{n_1} +\frac{1}{q_1}, \quad \frac{1}{p_1'} = \frac{\alpha}{m_1} +\frac{1}{p_1}, \quad \theta_1 =\frac{\alpha}{m_1} -\frac{\alpha}{p_1},
\]
and 
\[
q_1 \geq 3, \quad n_1 \in (q_1, \infty) \quad \text{or}\quad \frac{1}{n_1} = \frac{\tau}{q_1}, \quad \tau\in (0,1).
\]
Here the last condition ensures the Sobolev embedding $W^{1, q_1}(\R^3) \subset L^{n_1}(\R^3)$. We see that condition $\frac{3}{\gamma_1}>b$ implies
\[
\frac{3}{\gamma_1} = 3-\frac{3(2+\alpha\tau)}{q_1}>b \quad \text{or}\quad q_1>\frac{3(2+\alpha\tau)}{3-b}.
\]
Let us choose 
\[
q_1=\frac{3(2+\alpha\tau)}{3-b} +\eps,
\] 
for some $0<\eps\ll 1$ to be chosen later. Since $1\leq b<2, 0<\alpha<4-2b$ and $0<\tau <1$, it is obvious that $q_1>3$. Moreover, by taking $\eps>0$ small enough, we see that $q_1<6$. In order to make $\theta_1>0$, we need
\[
\theta_1=\frac{\alpha}{m_1}-\frac{\alpha}{p_1} = 1- \frac{\alpha+2}{p_1}>0 \quad \text{or} \quad \frac{2}{p_1}<\frac{2}{\alpha+2}.
\]
Since $(p_1, q_1)$ is Schr\"odinger admissible, it is equivalent to show
\[
\frac{3}{2}-\frac{3}{q_1}<\frac{2}{\alpha+2}. 
\]
It is then equivalent to
\[
3\left[8-4b-2b\alpha -\alpha \tau (2+3\alpha) \right] -\eps (3-b)(2+3\alpha)>0.
\]
Since $0<\eps\ll 1$, it is enough to show $f(\tau):=8-4b-2b\alpha -\alpha \tau (2+3\alpha)>0$. Note that $f(0)>0$ provided $0<\alpha <\frac{4-2b}{b}$ and $f(1)>0$ provided $0<\alpha <\frac{4-2b}{3}$. Thus, by choosing $\tau$ closed to 0, we see that $f(\tau)>0$ for $0<\alpha<\frac{4-2b}{b}$. Therefore, we get
\begin{align}
A_1 \lesssim T^{\theta_1} \|\scal{\nabla} u\|^\alpha_{S(L^2, I)} \|v\|_{S(L^2,I)}, \label{non-est-1-d3-prof1}
\end{align}
provided $\eps, \tau >0$ are taken small enough and
\[
1\leq b<2, \quad 0<\alpha <\frac{4-2b}{b}.
\]
The term $B_{11}$ is treated similarly as for $A_1$ by using the fractional chain rule. We obtain
\begin{align}
B_{11} \lesssim T^{\theta_1} \|\scal{\nabla} u\|^\alpha_{S(L^2, I)} \|\nabla u\|_{S(L^2,I)}, \label{non-est-1-d3-prof2}
\end{align}
provided $\eps, \tau >0$ is taken small enough and
\[
1\leq b<2, \quad 0<\alpha <\frac{4-2b}{b}.
\]
We next bound
\begin{align*}
B_{12}\leq \||x|^{-b-1}|u|^\alpha u\|_{L^{p_1'}_t(I, L^{q_1'}_x(B))} &\lesssim \||x|^{-b-1}\|_{L^{\gamma_1}_x(B)} \||u|^\alpha u\|_{L^{p_1'}_t(I, L^{\upsilon_1}_x)}\\
&\lesssim \|u\|^\alpha_{L^{m_1}_t(I, L^{n_1}_x)} \|u\|_{L^{p_1}_t(I, L^{n_1}_x)} \\
&\lesssim T^{\theta_1}\|\scal{\nabla} u\|^{\alpha+1}_{L^{p_1}_t(I, L^{q_1}_x)}, 
\end{align*}
provided that $(p_1, q_1)\in S$ and 
\[
\frac{1}{q_1'}=\frac{1}{\gamma_1} + \frac{1}{\upsilon_1}, \quad \frac{3}{\gamma_1}>b+1, \quad \frac{1}{\upsilon_1}=\frac{\alpha+1}{n_1}, \quad \frac{1}{p_1'}=\frac{\alpha}{m_1}+\frac{1}{p_1}, \quad \theta_1=\frac{\alpha}{m_1}-\frac{\alpha}{p_1},
\]
and
\[
q_1 \geq 3, \quad n_1 \in (q_1, \infty) \quad \text{or} \quad \frac{1}{n_1}=\frac{\tau}{q_1}, \quad \tau \in (0,1). 
\]
We see that
\[
\frac{3}{\gamma_1}= 3-\frac{3(1+(\alpha+1)\tau)}{q_1}>b+1 \quad \text{or}\quad q_1>\frac{3(1+(\alpha+1)\tau)}{2-b}.
\]
Let us choose 
\[
q_1=\frac{3(1+(\alpha+1)\tau)}{2-b}+\eps,
\] 
for some $0<\eps \ll 1$ to be determined later. Since we are considering $1\leq b<\frac{3}{2}$, by choosing $\tau$ closed to 0 and taking $\eps>0$ small enough, we can check that $3<q_1<6$. It remains to show $\theta_1>0$. As above, we need $\frac{2}{p_1}<\frac{2}{\alpha+2}$, and it is equivalent to
\[
\frac{3}{2}-\frac{3}{q_1}<\frac{2}{\alpha+2}.
\]
It is in turn equivalent to 
\[
3\left[6-4b+\alpha(1-2b) -(\alpha+1)\tau(2+3\alpha) \right] - \eps (2-b)(2+3\alpha)>0.
\]
Since $0<\eps \ll 1$, it is enough to show $g(\tau):= 6-4b +\alpha(1-2b) -(\alpha+1) \tau (2+3\alpha)>0$. Note that $g(0)>0$ provided $0<\alpha <\frac{6-4b}{2b-1}$. Thus, by choosing $\tau$ closed to 0, we see that $g(\tau)>0$ for $0<\alpha<\frac{6-4b}{2b-1}$. Therefore,
\begin{align}
B_{12} \lesssim T^{\theta_1} \|\scal{\nabla} u\|^{\alpha+1}_{S(L^2,I)},\label{non-est-1-d3-prof3}
\end{align}
provided $\eps, \tau >0$ are small enough and
\[
1\leq b <\frac{3}{2}, \quad 0<\alpha <\frac{6-4b}{2b-1}.
\]
\indent \underline{\bf On $B^c$.} Let us choose the following Schr\"odinger admissible pair
\[
p_2=\frac{4(\alpha+2)}{\alpha}, \quad q_2=\frac{3(\alpha+2)}{3+\alpha}.
\]
Let $m_2, n_2$ be such that
\begin{align}
\frac{1}{q_2'}=\frac{\alpha}{n_2} + \frac{1}{q_2}, \quad \frac{1}{p_2'}=\frac{\alpha}{m_2} + \frac{1}{p_2}. \label{choose m_2 n_2}
\end{align}
A direct computation shows
\[
\theta_2:=\frac{\alpha}{m_2}-\frac{\alpha}{p_2}=1-\frac{\alpha}{4}>0.
\]
Note that in our consideration $1\leq b<\frac{3}{2}, 0<\alpha<\frac{6-4b}{2b-1}$, we always have $\alpha <4$. Moreover, it is easy to check that
\[
\frac{1}{n_2}=\frac{1}{q_2}-\frac{1}{3}.
\]
It allows us to use the Sobolev embedding $W^{1, q_2}(\R^3) \subset L^{n_2}(\R^3)$. By H\"older inequality with $(\ref{choose m_2 n_2})$,
\begin{align}
A_2\leq \||x|^{-b}|u|^\alpha v\|_{L^{p_2'}_t(I, L^{q_2'}_x(B^c))} &\lesssim \||x|^{-b}\|_{L^\infty_x(B^c)} \||u|^\alpha v\|_{L^{p_2'}_t(I, L^{q_2'}_x)} \nonumber \\
&\lesssim \|u\|^\alpha_{L^{m_2}_t(I, L^{n_2}_x)}\|v\|_{L^{p_2}_t(I, L^{q_2}_x)} \nonumber \\
&\lesssim T^{\theta_2} \|\scal{\nabla} u\|^\alpha_{L^{p_2}_t(I, L^{q_2}_x)} \|v\|_{L^{p_2}_t(I, L^{q_2}_x)}.\nonumber
\end{align}
We thus get
\begin{align}
A_2 \lesssim T^{\theta_2} \|\scal{\nabla} u\|^\alpha_{S(L^2, I)} \|v\|_{S(L^2, I)}.\label{non-est-1-d3-prof4}
\end{align}
The term $B_{21}$ is treated similarly by using the fractional chain rule, and we obtain
\begin{align}
B_{21} \lesssim T^{\theta_2} \|\scal{\nabla} u\|^\alpha_{S(L^2, I)} \|\nabla u\|_{S(L^2, I)}.\label{non-est-1-d3-prof5}
\end{align}
Finally, we estimate
\begin{align}
B_{22}\leq \||x|^{-b-1}|u|^\alpha u\|_{L^{p_2'}_t(I, L^{q_2'}_x(B^c))} &\lesssim \||x|^{-b-1}\|_{L^{3}_x(B^c)} \|u\|^\alpha_{L^{m_2}_t(I, L^{n_2}_x)} \|u\|_{L^{p_2}_t(I, L^{n_2}_x)} \nonumber \\
&\lesssim T^{\theta_2}\|\scal{\nabla} u\|^{\alpha+1}_{L^{p_2}_t(I, L^{q_2}_x)}. \nonumber
\end{align}
This implies
\begin{align}
B_{22}\lesssim T^{\theta_2}\|\scal{\nabla} u\|^{\alpha+1}_{S(L^2, I)}. \label{non-est-1-d3-prof6}
\end{align}
Collecting \eqref{non-est-1-d3-prof1}--\eqref{non-est-1-d3-prof6}, we complete the proof.
\end{proof}

\begin{lemma} \label{lem-non-est-d2}
Let $d=2$. Let $0<b<1$ and $0<\alpha<\infty$ and $I=[0,T]$. Then there exists $\theta_1, \theta_2>0$ such that
\begin{align}
\||x|^{-b} |u|^\alpha v\|_{S'(L^2,I)} &\lesssim \left(T^{\theta_1}+T^{\theta_2}\right) \|\scal{\nabla} u\|^\alpha_{S(L^2,I)} \|v\|_{S(L^2, I)}, \label{non-est-1-d2} \\
\|\nabla(|x|^{-b} |u|^\alpha u)\|_{S'(L^2,I)} & \lesssim \left(T^{\theta_1}+T^{\theta_2}\right) \|\scal{\nabla} u\|^{\alpha+1}_{S(L^2, I)}. \label{non-est-2-d2}
\end{align}
\end{lemma}

\begin{remark} \label{rem-non-est-d2}
In \cite{Guzman}, Guzm\'an proved this result with $\theta_1 =\theta_2$ under the assumption $0<b<\frac{2}{3}$. Here we extend it to $0<b<1$. 
\end{remark}

\begin{remark} \label{rem-non-est-d1}
By using Strichartz estimate, we can not obtain a similar result as in Lemma \ref{lem-non-est-d-geq3}, Lemma \ref{lem-non-est-d3} and Lemma \ref{lem-non-est-d2} for the case $d=1$. The reason for this is the singularity $|x|^{-b-1}$ on $B$. To bound this term in a Lebesgue space $L^{\gamma}$ with $1\leq \gamma \leq \infty$, we need
\[
\frac{d}{\gamma} > b+1. 
\]
This implies that we need at least $d>b+1$, which does not hold when $d=1$. 
\end{remark}

\noindent \textit{Proof of Lemma \ref{lem-non-est-d2}.} 
We continue to use the notations $A_1, A_2, B_{11}, B_{12}, B_{21}$ and $B_{22}$ introduced in the proof of Lemma \ref{lem-non-est-d-geq3}. \newline
\indent \underline{\bf On $B$.} By H\"older inequality and Remark $\ref{rem singularity bound}$, 
\begin{align*}
A_1 \leq \||x|^{-b} |u|^\alpha v\|_{L^{p_1'}_t(I, L^{q_1'}_x(B))} &\lesssim \||x|^{-b}\|_{L^{\gamma_1}_x(B)} \||u|^\alpha v\|_{L^{p_1'}_t(I, L^{\upsilon_1}_x)} \\
&\lesssim \|u\|^\alpha_{L^{m_1}_t(I, L^{n_1}_x)} \|v\|_{L^\infty_t(I, L^2_x)} \\
&\lesssim \|\scal{\nabla}u\|^\alpha_{L^{m_1}_t(I, L^2_x)} \|v\|_{L^\infty_t(I, L^2_x)} \\
&\lesssim T^{\theta_1} \|\scal{\nabla} u\|^\alpha_{L^\infty_t(I, L^2_x)} \|v\|_{L^\infty_t(I, L^2_x)},
\end{align*}
provided that $(p_1, q_1)\in S$ and 
\[
\frac{1}{q_1'}=\frac{1}{\gamma_1} +\frac{1}{\upsilon_1}, \quad \frac{2}{\gamma_1}>b, \quad \frac{1}{\upsilon_1} = \frac{\alpha}{n_1} + \frac{1}{2}, \quad \frac{1}{p_1'}=\frac{\alpha}{m_1}=\theta_1,
\]
and
\[
n_1 \in (2, \infty) \quad \text{or} \quad \frac{1}{n_1} =\frac{\tau}{2}, \quad \tau \in (0,1).
\]
The last condition allows us to use the Sobolev embedding $W^{1,2}(\R^2) \subset L^{n_1}(\R^2)$. The condition $\frac{2}{\gamma_1}>b$ implies
\[
\frac{2}{\gamma_1}=1-\frac{2}{q_1} - \alpha \tau >b \quad \text{or} \quad \frac{2}{q_1}<1-b-\alpha \tau.
\]
Note that since $0<b<1$, by taking $\tau>0$ small enough, we see that $1-b-\alpha \tau>0$. Let us choose 
\[
q_1=\frac{2}{1-b-\alpha \tau} +\eps,
\]
for some $0<\eps\ll 1$ to be chosen later. It is obvious that $2<q_1<\infty$ and $\theta_1>0$. Therefore, we obtain
\begin{align}
A_1 \lesssim T^{\theta_1} \|\scal{\nabla} u\|^\alpha_{S(L^2,I)} \|v\|_{S(L^2,I)}. \label{non-est-1-d2-prof1}
\end{align}
The term $B_{11}$ is again treated similarly as for $A_1$ above using the fractional chain rule. We get
\begin{align}
B_{11} \lesssim T^{\theta_1} \|\scal{\nabla} u\|^\alpha_{S(L^2,I)} \|\nabla u\|_{S(L^2,I)}. \label{non-est-1-d2-prof2}
\end{align}
We continue to bound
\begin{align*}
B_{12} \leq \||x|^{-b-1} |u|^\alpha u\|_{L^{p_1'}_t(I, L^{q_1'}_x(B))} &\lesssim \||x|^{-b-1}\|_{L^{\gamma_1}_x(B)} \||u|^\alpha u\|_{L^{p_1'}_t(I, L^{\upsilon_1}_x)} \\
&\lesssim	 \|u\|^\alpha_{L^{m_1}_t(I, L^{n_1}_x)} \|u\|_{L^\infty_t(I, L^{n_1}_x)} \\
&\lesssim  \|\scal{\nabla} u\|^\alpha_{L^{m_1}_t(I, L^2_x)} \|\scal{\nabla}u\|_{L^\infty_t(I, L^2_x)} \\
&\lesssim T^{\theta_1} \|\scal{\nabla} u\|^{\alpha}_{L^\infty_t(I, L^2_x)} \|\scal{\nabla} u\|_{L^\infty_t(I, L^2_x)},
\end{align*}
provided that $(p_1, q_1)\in S$ and 
\[
\frac{1}{q_1'} =\frac{1}{\gamma_1} +\frac{1}{\upsilon_1}, \quad \frac{2}{\gamma_1}>b+1, \quad \frac{1}{\upsilon_1} =\frac{\alpha+1}{n_1}, \quad \frac{1}{p_1'}=\frac{\alpha}{m_1}=\theta_1,
\]
and 
\[
n_1 \in (2, \infty) \quad \text{or} \quad \frac{1}{n_1} =\frac{\tau}{2}, \quad \tau \in (0,1).
\]
The condition $\frac{2}{\gamma_1}>b+1$ implies
\[
\frac{2}{\gamma_1} = 2-\frac{2}{q_1}- (\alpha+1)\tau > b+1 \quad \text{or} \quad \frac{2}{q_1}<1-b-(\alpha+1)\tau.
\]
Since $0<b<1$, by choosing $\tau$ closed to 0, we see that $1-b-(\alpha+1)\tau>0$. Let us choose 
\[
q_1=\frac{2}{1-b-(\alpha+1)\tau} +\eps,
\] 
for some $0<\eps \ll 1$ to be chosen later. It is obvious that $2<q_1<\infty$ and $\theta_1>0$. Thus, we obtain
\begin{align}
B_{12} \lesssim T^{\theta_1} \|\scal{\nabla} u\|^{\alpha+1}_{S(L^2,I)}. \label{non-est-1-d2-prof3}
\end{align} 
\indent \underline{\bf On $B^c$.} Let us choose the following Schr\"odinger admissible pair
\[
p_2 = \frac{2(\alpha+2)}{\alpha}, \quad q_2 =\alpha+2.
\]
It is easy to see that $\frac{1}{q_2'} = \frac{\alpha+1}{q_2}$. By H\"older's inequality,
\begin{align*}
A_2 \leq \||x|^{-b} |u|^\alpha v\|_{L^{p_2'}_t(I, L^{q_2'}_x(B^c))} &\lesssim \||x|^{-b}\|_{L^{\infty}_x(B^c)} \||u|^\alpha v\|_{L^{p_2'}_t(I, L^{q_2'}_x)} \\
&\lesssim \|u\|^\alpha_{L^{m_2}_t(I, L^{q_2}_x)}\|v\|_{L^{p_2}_t(I, L^{q_2}_x)} \\
&\lesssim \|\scal{\nabla} u\|^\alpha_{L^{m_2}_t(I, L^2_x)} \|v\|_{L^{p_2}_t(I, L^{q_2}_x)} \\
&\lesssim T^{\theta_2} \|\scal{\nabla} u\|^\alpha_{L^\infty_t(I, L^2_x)} \|v\|_{L^{p_2}_t(I, L^{q_2}_x)},
\end{align*}
where
\[
\frac{1}{p'_2}=\frac{\alpha}{m_2}+\frac{1}{p_2}, \quad \theta_2=\frac{\alpha}{m_2}=\frac{2}{\alpha+2}>0.
\]
We thus get
\begin{align}
A_2 \lesssim T^{\theta_2} \|\scal{\nabla} u\|^\alpha_{S(L^2, I)} \|v\|_{S(L^2, I)}. \label{non-est-1-d2-prof4}
\end{align}
By using the fractional chain rule and estimating as for $A_2$, we get
\begin{align}
B_{21} \lesssim T^{\theta_2} \|\scal{\nabla} u\|^\alpha_{S(L^2, I)} \|\nabla u\|_{S(L^2, I)}. \label{non-est-1-d2-prof5}
\end{align}
Finally, we bound
\begin{align*}
B_{22} \leq \||x|^{-b-1} |u|^\alpha u\|_{L^{p_2'}_t(I, L^{q_2'}_x(B^c))} &\lesssim \||x|^{-b-1}\|_{L^{\infty}_x(B^c)} \||u|^\alpha u\|_{L^{p_2'}_t(I, L^{q_2'}_x)} \\
&\lesssim \|u\|^\alpha_{L^{m_2}_t(I, L^{q_2}_x)} \|u\|_{L^{p_2}_t(I, L^{q_2}_x)} \\
&\lesssim \|\scal{\nabla}u\|^\alpha_{L^{m_2}_t(I, L^2_x)} \|u\|_{L^{p_2}_t(I, L^{q_2}_x)} \\
&\lesssim T^{\theta_2} \|\scal{\nabla}u\|^\alpha_{L^\infty_t(I, L^2_x)} \|u\|_{L^{p_2}_t(I, L^{q_2}_x)} \\
&\lesssim T^{\theta_2} \|\scal{\nabla} u\|^\alpha_{L^\infty_t(I, L^2_x)} \|\scal{\nabla}u\|_{L^{p_2}_t(I, L^{q_2}_x)}.
\end{align*}
Where $m_2, \theta_2$ are as in term $A_2$. Thus, we obtain
\begin{align}
B_{22} \lesssim T^{\theta_2} \|\scal{\nabla} u\|^{\alpha+1}_{S(L^2, I)}. \label{non-est-1-d2-prof6}
\end{align}
Collecting \eqref{non-est-1-d2-prof1}--\eqref{non-est-1-d2-prof6}, we complete the proof. 
\hfill $\Box$


\indent We are now able to prove Theorem \ref{theorem local existence}. From now on, we denote for any spacetime slab $I \times \R^d$,
\begin{align}
\|u\|_{S(I)}:= \|\scal{\nabla} u\|_{S(L^2, I)} =  \|u\|_{S(L^2, I)} + \|\nabla u\|_{S(L^2,I)}. \label{define h1 strichartz norm}
\end{align}
\paragraph{\bf Proof of Theorem \ref{theorem local existence}.}
We follow the standard argument (see e.g. \cite[Chapter 4]{Cazenave}). Let
\[
X= \Big\{ u \in C_t(I, H^1_x) \cap L^p_t(I, W^{1, q}_x), \forall (p,q) \in S \ | \ \|u\|_{S(I)} \leq M \Big\},
\]
equipped with the distance
\[
d(u,v) = \|u-v\|_{S(L^2, I)},
\]
where $I=[0,T]$ and $T, M>0$ to be chosen later. By the Duhamel formula, it suffices to prove that the functional
\[
\Phi(u)(t) =e^{it\Delta}u_0 + i\mu\int_0^t e^{i(t-s)\Delta} |x|^{-b} |u(s)|^\alpha u(s) ds
\]
is a contraction on $(X,d)$. By Strichartz estimates, we have
\begin{align*}
\|\Phi(u)\|_{S(I)} &\lesssim \|u_0\|_{H^1_x} + \||x|^{-b} |u|^\alpha u\|_{S'(L^2, I)} + \|\nabla (|x|^{-b}|u|^\alpha u)\|_{S'(L^2, I)}, \\
\|\Phi(u) -\Phi(v)\|_{S(L^2,I)} &\lesssim \||x|^{-b}(|u|^\alpha u -|v|^\alpha v)\|_{S'(L^2,I)}.
\end{align*}
Applying Lemmas \ref{lem-non-est-d-geq3}, \ref{lem-non-est-d3}, \ref{lem-non-est-d2}, we get for some $\theta_1, \theta_2>0$,
\begin{align*}
\|\Phi(u) \|_{S(I)} &\lesssim \|u_0\|_{H^1_x} + \left(T^{\theta_1} + T^{\theta_2}\right) \|u\|^{\alpha+1}_{S(I)}, \\
\|\Phi(u) - \Phi(v)\|_{S(L^2, I)}  &\lesssim \left(T^{\theta_1} + T^{\theta_2}\right) \left(\|u\|^\alpha_{S(I)} + \|v\|^\alpha_{S(I)}\right) \|u-v\|_{S(L^2, I)}.
\end{align*}
This shows that for $u,v \in X$, there exists $C>0$ independent of $T$ and $u_0 \in H^1_x$ such that
\begin{align*}
\|\Phi(u)\|_{S(I)} &\leq C\|u_0\|_{H^1_x} + C\left(T^{\theta_1}+T^{\theta_2}\right) M^{\alpha+1}, \\
d(\Phi(u), \Phi(v)) &\leq C \left(T^{\theta_1}+T^{\theta_2}\right) M^\alpha d(u,v). 
\end{align*}
If we set $M= 2C \|u_0\|_{H^1_x}$ and choose $T>0$ so that
\[
C\left(T^{\theta_1}+T^{\theta_2}\right) M^\alpha \leq \frac{1}{2},
\]
then $\Phi$ is a strict contraction on $(X,d)$. The proof is complete.
\hfill $\Box$

\section{Pseudo-conformal conservation law} 
\setcounter{equation}{0}
\label{S4}

In this section, we firstly derive the virial identity and then use it to show the pseudo-conformal conservation law related to the defocusing (INLS). The proof is based on the standard technique (see e.g. \cite{Cazenave, Tao}). Given a smooth real valued function $a$, we define the virial potential by
\begin{align}
V_a(t):= \int a(x)|u(t,x)|^2 dx. \label{virial potential} 
\end{align}
By a direct computation, we have the following result (see e.g. \cite[Lemma 5.3]{TVZ} for the proof).
\begin{lemma}[\cite{TVZ}]  \label{lem derivative virial potential}
If $u$ is a smooth-in-time and Schwartz-in-space solution to 
\[
i\partial_t u +\Delta u = N(u),
\]
with $N(u)$ satisfying $\ima{(N(u)\overline{u})}=0$, then we have
\begin{align}
\frac{d}{dt} V_a (t)= 2 \int_{\R^d}\nabla a(x) \cdot  \ima	{(\overline{u}(t,x) \nabla u(t,x))} dx, \label{first derivative viral potential}
\end{align}
and
\begin{equation}
\begin{aligned}
\frac{d^2}{dt^2} V_a(t) = -\int \Delta^2 a(x) |u(t,x)|^2  dx & +  4 \sum_{j,k=1}^d \int \partial^2_{jk} a(x) \rea{(\partial_k u(t,x) \partial_j \overline{u}(t,x))} dx  \\
&+ 2\int \nabla a(x)\cdot \{N(u), u\}_p(t,x) dx, 
\end{aligned} \label{second derivative virial potential} 
\end{equation}
where $\{f, g\}_p :=\rea{(f\nabla \overline{g} - g \nabla \overline{f})}$ is the momentum bracket.
\end{lemma}
 
\begin{corollary}\label{coro derivative virial potential}
If $u$ is a smooth-in-time and Schwartz-in-space solution to the defocusing \emph{(INLS)}, then we have 
\begin{equation}
\begin{aligned}
\frac{d^2}{dt^2} V_a(t) &= -\int \Delta^2 a(x) |u(t,x)|^2  dx  +  4 \sum_{j,k=1}^d \int \partial^2_{jk} a(x) \rea{(\partial_k u(t,x) \partial_j \overline{u}(t,x))} dx  \\
&\mathrel{\phantom{=}}+\frac{2\alpha}{\alpha+2} \int \Delta a(x) |x|^{-b} |u(t,x)|^{\alpha+2} dx -\frac{4}{\alpha+2} \int \nabla a(x) \cdot \nabla(|x|^{-b}) |u(t,x)|^{\alpha+2} dx. 
\end{aligned} \label{second derivative virial potential application} 
\end{equation}
\end{corollary}

\begin{proof}
Applying Lemma $\ref{lem derivative virial potential}$ with $N(u)=F(x,u)=|x|^{-b} |u|^\alpha u$. Note that
\[
\{N(u), u\}_p = -\frac{\alpha}{\alpha+2} \nabla(|x|^{-b} |u|^{\alpha+2}) -\frac{2}{\alpha+2} \nabla(|x|^{-b}) |u|^{\alpha+2}.
\] 
\end{proof}

We now have the following virial identity for the defocusing (INLS). 
\begin{proposition} \label{prop virial identity}
Let $u_0 \in H^1(\R^d)$ be such that $|x|u_0 \in L^2(\R^d)$ and $u$ the corresponding global solution to the defocusing \emph{(INLS)}. Then $|x| u \in C(\R, L^2(\R^d))$. Moreover, for any $t\in \R$,
\begin{align}
\frac{d^2}{dt^2} \|x u(t)\|^2_{L^2_x} = 16 E(u_0) + 4(d\alpha+2b-4)G(t), \label{virial identity}
\end{align}
where $G$ is given in $(\ref{define G})$.
\end{proposition}

\begin{proof}
The first claim follows from the standard approximation argument, we omit the proof and refer the reader to \cite[Proposition 6.5.1]{Cazenave} for more details. It remains to show $(\ref{virial identity})$. Applying Corollary $\ref{coro derivative virial potential}$ with $a(x)=|x|^2$, we have
\begin{align*}
\frac{d^2}{dt^2}V_a(t)=\frac{d^2}{dt^2} \|x u(t)\|^2_{L^2_x} &= 8 \|\nabla u(t)\|^2_{L^2_x} + 4(d\alpha +2b)G(t) \\
&= 16 E(u(t)) + 4(d\alpha +2b-4)G(t).
\end{align*}
The result follows by using the conservation of energy.
\end{proof}

An application of the virial identity is the following ``pseudo-conformal conservation law" for the defocusing (INLS). 
\begin{lemma}
Let $u_0 \in H^1(\R^d)$ be such that $|x| u_0 \in L^2(\R^d)$ and $u$ the corresponding global solution to the defocusing \emph{(INLS)}. Then for any $t\in \R$,
\begin{align}
\|(x+2it\nabla) u(t)\|^2_{L^2_x} +8t^2 G(t) = \|x u_0\|^2_{L^2_x} + 4(4-2b-d\alpha)\int_0^t s G(s) ds. \label{pseudoconformal conservation law}
\end{align}
\end{lemma}

\begin{proof}
Set 
\[
f(t) := \|(x+2it\nabla) u(t)\|^2_{L^2_x} + 8t^2 G(t).
\]
By $(\ref{first derivative viral potential})$, we see that
\begin{align*}
\|(x+2it\nabla) u(t)\|^2_{L^2_x} &= \|x u(t)\|^2_{L^2_x} + 4t^2 \|\nabla u(t)\|^2_{L^2_x} -4t \int \ima{(\overline{u}(t,x) x \cdot \nabla u(t,x) )}dx \\
&= \|x u(t)\|^2_{L^2_x} + 4t^2 \|\nabla u(t)\|^2_{L^2_x} - t \frac{d}{dt} \|x u(t)\|^2_{L^2_x}.
\end{align*}
Thus, the conservation of energy implies
\[
f(t) = \|x u(t)\|^2_{L^2_x} +8t^2 E(u(t)) -t\frac{d}{dt} \|x u(t)\|^2_{L^2_x} = \|x u(t)\|^2_{L^2_x} + 8t^2 E(u_0) -t\frac{d}{dt} \|x u(t)\|^2_{L^2_x}.
\]
Applying $(\ref{virial identity})$, we get
\[
f'(t) = \frac{d}{dt} \|x u(t)\|^2_{L^2_x} + 16 t E(u_0) - \frac{d}{dt} \|x u(t)\|^2_{L^2_x} -t \frac{d^2}{dt^2} \|x u(t)\|^2_{L^2_x} = 4(4-2b-d\alpha) t G(t).
\]
Taking integration on $(0,t)$, we obtain $(\ref{pseudoconformal conservation law})$. 
\end{proof}
\begin{remark} \label{rem extension pseudoconformal conservation law}
This result extends the pseudo-conformal conservation law for the classical (i.e. $b=0$) nonlinear Schr\"odinger equation (see e.g. \cite[Theorem 7.2.1]{Cazenave}). Note that $(\ref{pseudoconformal conservation law})$ is a real conservation law only when $\alpha =\frac{4-2b}{d}$.
\end{remark}

\begin{remark} \label{rem pseudoconformal conservation law}
It is easy to see that if $t\ne 0$, then 
\begin{align}
(x+2it\nabla) u(t,x) = 2it e^{i\frac{|x|^2}{4t}} \nabla\Big(e^{-i\frac{|x|^2}{4t}} u(t,x)\Big), \label{change variable}
\end{align}
and
\[
\|(x+2it\nabla) u(t)\|^2_{L^2_x} = 4t^2 \Big\|\nabla\Big(e^{-i\frac{|x|^2}{4t}} u(t,x)\Big) \Big\|^2_{L^2_x}.
\]
Therefore, if we set
\begin{align}
v(t,x):= e^{-i\frac{|x|^2}{4t}} u(t,x), \label{define v}
\end{align}
then
\begin{align*}
\|(x+2it\nabla) u(t)\|^2_{L^2_x} = 4t^2 \|\nabla v(t)\|^2_{L^2_x},
\end{align*}
and $(\ref{pseudoconformal conservation law})$ becomes
\begin{align}
8t^2 E(v(t)) = \|x u_0\|^2_{L^2_x} +4(4-2b-d\alpha) \int_0^t s G(s)ds. \label{equivalence pseudoconformal conservation law}
\end{align}
\end{remark}

\begin{remark} \label{rem change variable}
Let $F(x,u) = |x|^{-b} |u|^\alpha u$. It follows from $(\ref{change variable})$ that
\begin{align}
|(x+2it\nabla) F(x,u)| = 2|t| \Big|\nabla\Big(e^{-i\frac{|x|^2}{4t}} F(x,u) \Big) \Big| = 2|t| |\nabla F(x,v)|, \label{change variable property 1}
\end{align}
where $v$ is given in $(\ref{define v})$. Using the facts $|v|=|u|$ and $2|t| |\nabla v| = |(x+2it\nabla) u|$, we also have
\begin{align}
\|v\|_{L^q_x} = \|u\|_{L^q_x}, \quad 2|t| \|\nabla v\|_{L^q_x} = \|(x+2it\nabla) u\|_{L^q_x}. \label{change variable property 2}
\end{align}
\end{remark}

\section{Decay of solutions} 
\setcounter{equation}{0}
\label{S5}

In this section, we will give the proof of the decaying property given in Theorem $\ref{theorem decay property}$. We follows the standard argument of Ginibre and Velo \cite{GV} (see also \cite[Chapter 7]{Cazenave}).
\paragraph{\bf Proof of Theorem \ref{theorem decay property}.} 
We have from \eqref{equivalence pseudoconformal conservation law} that
\begin{align}
8t^2 E(v(t))=8t^2 \Big(\frac{1}{2} \|\nabla v(t)\|^2_{L^2_x} + G(t) \Big) = \|x u_0\|^2_{L^2_x}  + 4(4-2b-d\alpha) \int_0^t s G(s) ds, \label{equivalence 2 pseudoconformal conservation law}
\end{align}
for all $t\in \R$, where $v$ is defined in $(\ref{define v})$. \newline
\indent If $\alpha \in [\alpha_\star, \alpha^\star)$, then $(\ref{equivalence 2 pseudoconformal conservation law})$ implies
\[
4t^2 \|\nabla v(t)\|^2_{L^2_x} \leq \|x u_0\|^2_{L^2_x}, 
\]
for all $t\in \R$. Hence, $\|\nabla v(t)\|_{L^2_x} \lesssim |t|^{-1}$ for $t \in \R \backslash \{0\}$. Using $(\ref{change variable property 2})$, Gagliardo-Nirenberg's inequality and the conservation of mass, we have
\begin{align*}
\|u(t)\|_{L^q_x} =\|v(t)\|_{L^q_x}  &\lesssim \|\nabla v(t)\|_{L^2_x}^{d\left(\frac{1}{2}-\frac{1}{q}\right)} \|v(t)\|^{1-d\left(\frac{1}{2}-\frac{1}{q} \right)}_{L^2_x}  \\
&\lesssim |t|^{-d\left(\frac{1}{2}-\frac{1}{q}\right)} \|u_0\|_{L^2_x}^{1-d\left(\frac{1}{2}-\frac{1}{q} \right)} \lesssim |t|^{-d\left(\frac{1}{2}-\frac{1}{q}\right)}.
\end{align*} 
This proves the first claim. \newline
\indent We now assume $\alpha \in (0, \alpha_\star)$.  Note that it suffices to show the decay for $|t|\geq 1$, the one for $|t|<1$ follows by H\"older's inequality and the conservations of mass and energy. Let us consider only the case $t\geq 1$, the case $t\leq-1$ is treated similarly. By taking $t=1$ in $(\ref{equivalence 2 pseudoconformal conservation law})$, we see that
\[
8E(v(1)) = \|x u_0\|^2_{L^2_x} + 4(4-2b-d\alpha) \int_0^1 s G(s) ds.
\]
Thus,
\[
8t^2 E(v(t)) = 8E(v(1)) + 4(4-2b-d\alpha) \int_1^t s G(s) ds. 
\]
This implies
\[
g(t):= t^2 G(t) \leq E(v(1)) +\frac{4-2b-d\alpha}{2} \int_1^t \frac{1}{s} g(s) ds.
\]
Applying Gronwall's inequality, we obtain
\[
g(t) \lesssim t^{\frac{4-2b-d\alpha}{2}}, \quad \text{hence} \quad G(t) \lesssim t^{\frac{-2b-d\alpha}{2}}.
\]
By $(\ref{equivalence 2 pseudoconformal conservation law})$, we have
\[
4t^2 \|\nabla v(t)\|^2_{L^2_x} \lesssim \|x u_0\|^2_{L^2_x} + 4(4-2b-d\alpha) \int_0^t s^{\frac{2-2b-d\alpha}{2}}ds \lesssim 1 + t^{\frac{4-2b-d\alpha}{2}},
\]
or
\[
\|\nabla v(t)\|_{L^2_x} \lesssim t^{-\frac{2b+d\alpha}{4}}. 
\]
By Gagliardo-Nirenberg's inequality, the conservation of mass and $(\ref{change variable property 2})$, we obtain
\begin{align*}
\|u(t)\|_{L^q_x} =\|v(t)\|_{L^q_x}  &\lesssim \|\nabla v(t)\|_{L^2_x}^{d\left(\frac{1}{2}-\frac{1}{q}\right)} \|v(t)\|^{1-d\left(\frac{1}{2}-\frac{1}{q} \right)}_{L^2_x}  \\
&\lesssim t^{-\frac{d(2b+d\alpha)}{4}\left(\frac{1}{2}-\frac{1}{q}\right)} \|u_0\|_{L^2_x}^{1-d\left(\frac{1}{2}-\frac{1}{q} \right)} \lesssim t^{-\frac{d(2b+d\alpha)}{4}\left(\frac{1}{2}-\frac{1}{q}\right)}.
\end{align*}
This completes the proof.
\hfill $\Box$

\section{Scattering in the weighted $L^2$ space} 
\setcounter{equation}{0}
\label{S6}

In this section, we will give the proof of the scattering in the weighted space $\Sigma$ given in Theorem $\ref{theorem scattering weighted space}$. To do this, we use the decay given in Theorem $\ref{theorem decay property}$ to obtain global bounds on the solution. The scattering property follows easily from the standard argument. We also give some comments in the case $\alpha\in (0,\alpha_\star)$ in the end of this section.\newline
\indent Let us introduce the following so-called Strauss exponent
\begin{align}
\alpha_0 := \frac{2-d-2b + \sqrt{d^2+12d+4 + 4b(b-2-d)}}{2d}, \label{strauss exponent}
\end{align}
which is the positive root to the following quadratic equation
\[
d\alpha^2 + (d-2+2b) \alpha + 2b-4 =0.
\]
\begin{remark} \label{rem strauss exponent}
It is easy to check that for $0<b<\min \{2, d\}$,
\[
\alpha_0<\frac{4-2b}{d}.
\]
\end{remark}

Note that when $b=0$, $\alpha_0$ is the classical Strauss exponent introduced in \cite{Tsutsumi} (see also \cite{CW, Cazenave}). Let us start with the following lemmas providing some estimates on the nonlinearity.

\begin{lemma}\label{lem scattering weighted space d geq 4}
Let $d\geq 4$, $b\in(0,2)$ and $\alpha \in [\alpha_\star, \alpha^\star)$. Then there exist $(p_1, q_1), (p_2, q_2) \in S$ satisfying $2\alpha+2>p_1, p_2$ and $q_1, q_2 \in \left(2, \frac{2d}{d-2}\right)$ such that
\begin{align}
\||x|^{-b} |u|^\alpha u \|_{S'(L^2, I)} &\lesssim \left(\|u\|^\alpha_{L^{m_1}_t(I, L^{q_1}_x)} + \|u\|^\alpha_{L^{m_2}_t(I, L^{q_2}_x)}\right) \|u\|_{S(L^2, I)}, \label{estimate nonlinear 1 scattering weighted space} \\
\|\nabla(|x|^{-b} |u|^\alpha u)\|_{S'(L^2, I)} &\lesssim \left(\|u\|^\alpha_{L^{m_1}_t(I, L^{q_1}_x)} + \|u\|^\alpha_{L^{m_2}_t(I, L^{q_2}_x)}\right) \|\nabla u\|_{S(L^2, I)}, \label{estimate nonlinear 2 scattering weighted space}
\end{align}
where $m_1 = \frac{\alpha p_1}{p_1-2}$ and $m_2 = \frac{\alpha p_2}{p_2-2}$. 
\end{lemma}

\begin{proof}
Let us bound
\[
\||x|^{-b} |u|^\alpha u \|_{S'(L^2, I)} \leq \||x|^{-b} |u|^\alpha u \|_{S'(L^2(B), I)} + \||x|^{-b} |u|^\alpha u \|_{S'(L^2(B^c), I)}=:A_1 + A_2,
\]
and 
\[
\nabla(|x|^{-b} |u|^\alpha u)\|_{S'(L^2, I)} \leq \nabla(|x|^{-b} |u|^\alpha u)\|_{S'(L^2(B), I)} + \nabla(|x|^{-b} |u|^\alpha u)\|_{S'(L^2(B^c), I)}=:B_1 + B_2,
\]
where 
\begin{align*}
B_1 &\leq \||x|^{-b} \nabla(|u|^\alpha u) \|_{S'(L^2(B), I)} + \||x|^{-b-1} |u|^\alpha u \|_{S'(L^2(B), I)} =: B_{11}+B_{12}, \\
B_2 &\leq \||x|^{-b} \nabla(|u|^\alpha u) \|_{S'(L^2(B^c), I)} + \||x|^{-b-1} |u|^\alpha u \|_{S'(L^2(B^c), I)} 	=: B_{21}+B_{22}.
\end{align*}
\indent \underline{\bf On $B$.} By H\"older's inequality and Remark $\ref{rem singularity bound}$,
\begin{align*}
A_1 \leq \||x|^{-b} |u|^\alpha u \|_{L^{p_1'}_t(I, L^{q_1'}_x(B))} &\lesssim \||x|^{-b}\|_{L^{\gamma_1}_x(B)} \||u|^\alpha u\|_{L^{p_1'}_t(I, L^{\upsilon_1}_x)} \\
&\lesssim \|u\|^\alpha_{L^{m_1}_t(I, L^{q_1}_x)} \|u\|_{L^{p_1}_t(I, L^{q_1}_x)},
\end{align*}
provided that $(p_1, q_1)\in S$ and
\[
\frac{1}{q_1'}=\frac{1}{\gamma_1}+\frac{1}{\upsilon_1}, \quad \frac{d}{\gamma_1}>b, \quad \frac{1}{\upsilon_1}=\frac{\alpha+1}{q_1}, \quad \frac{1}{p_1'}=\frac{\alpha}{m_1}+\frac{1}{p_1}.
\]
These conditions imply
\[
\frac{d}{\gamma_1} = d-\frac{d(\alpha+2)}{q_1}>b, \quad \frac{\alpha}{m_1} = 1-\frac{2}{p_1}.
\]
Let us choose 
\begin{align}
q_1=\frac{d(\alpha+2)}{d-b}+\eps, \label{choice q_1}
\end{align}
for some $0<\eps \ll 1$ to be chosen later. Since we are considering $d\geq 4, b\in (0,2)$ and $\alpha \in [\alpha_\star, \alpha^\star)$, it is easy to check that $q_1 \in \left(2, \frac{2d}{d-2}\right)$ provided that $\eps>0$ is taken small enough. We thus get
\begin{align}
A_1 \lesssim \|u\|^\alpha_{L^{m_1}_t(I, L^{q_1}_x)} \|u\|_{S(L^2, I)}. \label{estimate nonlinear 1 scattering weighted space proof 1}
\end{align}
The term $B_{11}$ is treated similarly by using the fractional chain rule, and we have
\begin{align}
B_{11} \lesssim \|u\|^\alpha_{L^{m_1}_t(I, L^{q_1}_x)} \|\nabla u\|_{S(L^2, I)}. \label{estimate nonlinear 2 scattering weighted space proof 1}
\end{align}
We next bound
\begin{align*}
B_{12} \leq \||x|^{-b-1} |u|^\alpha u \|_{L^{p_1'}_t(I, L^{q_1'}_x(B))} &\lesssim \||x|^{-b-1}\|_{L^{\gamma_1}_x(B)} \||u|^\alpha u\|_{L^{p_1'}_t(I, L^{\upsilon_1}_x)} \\
&\lesssim \|u\|^\alpha_{L^{m_1}_t(I, L^{q_1}_x)} \|u\|_{L^{p_1}_t(I, L^{n_1}_x)} \\
&\lesssim \|u\|^\alpha_{L^{m_1}_t(I, L^{q_1}_x)} \|\nabla u\|_{L^{p_1}_t(I, L^{q_1}_x)},
\end{align*}
provided
\begin{align*}
\frac{1}{q_1'}=\frac{1}{\gamma_1}+\frac{1}{\upsilon_1}, \quad \frac{d}{\gamma_1}>b+1, \quad \frac{1}{\upsilon_1} = \frac{\alpha}{q_1} + \frac{1}{n_1}, \quad \frac{1}{p_1'}=\frac{\alpha}{m_1}+\frac{1}{p_1}, 
\end{align*}
and
\[
q_1 <d, \quad  \frac{1}{n_1} = \frac{1}{q_1}-\frac{1}{d}.
\]
Here the last condition allows us to use the homogeneous Sobolev embedding $\dot{W}^{1, q_1}(\R^d)\subset L^{n_1}(\R^d)$. Note that by taking $\eps>0$ small enough, the condition $q_1<d$ implies $\alpha <d-b-2$ which is true for $d\geq 4$ and $\alpha \in [\alpha_\star, \alpha^\star)$. We then have
\[
\frac{d}{\gamma_1} = d-\frac{d(\alpha+2)}{q_1} +1 >b+1, \quad \frac{\alpha}{m_1}=1-\frac{2}{p_1}. 
\]
Therefore, by choosing $q_1$ as in $(\ref{choice q_1})$, we obtain
\begin{align}
B_{12} \lesssim \|u\|^\alpha_{L^{m_1}_t(I, L^{q_1}_x)} \|\nabla u\|_{S(L^2, I)}. \label{estimate nonlinear 2 scattering weighted space proof 2}
\end{align}
\indent \underline{\bf On $B^c$.} By H\"older's inequality and Remark $\ref{rem singularity bound}$,
\begin{align*}
A_2 \leq \||x|^{-b}|u|^\alpha u\|_{L^{p_2'}_t(I, L^{q_2'}_x(B^c))} &\lesssim \||x|^{-b}\|_{L^{\gamma_2}_x(B^c)} \||u|^\alpha u\|_{L^{p_2'}_t(I, L^{\upsilon_2}_x)} \\
&\lesssim \|u\|^\alpha_{L^{m_2}_t(I, L^{q_2}_x)}\|u\|_{L^{p_2}_t(I, L^{q_2}_x)}, 
\end{align*}
provided that $(p_2, q_2) \in S$ and 
\[
\frac{1}{q_2'}=\frac{1}{\gamma_2}+\frac{1}{\upsilon_2}, \quad\frac{d}{\gamma_2}<b, \quad \frac{1}{\upsilon_2}=\frac{\alpha+1}{q_2}, \quad\frac{1}{p_2'}=\frac{\alpha}{m_2}+\frac{1}{p_2}.
\]
These conditions imply
\[
\frac{d}{\gamma_2}=d-\frac{d(\alpha+2)}{q_2}<b, \quad\frac{\alpha}{m_2}=1-\frac{2}{p_2}.
\]
Let us choose 
\begin{align}
q_2 = \frac{d(\alpha+2)}{d-b}-\eps, \label{choice q_2}
\end{align} 
for some $0<\eps \ll 1$ to be chosen later. By taking $\eps>0$ small enough, we see that $q_1 \in \left(2, \frac{2d}{d-2}\right)$. We thus obtain
\begin{align}
A_2 \lesssim \|u\|^\alpha_{L^{m_2}_t(I, L^{q_2}_x)}\|u\|_{S(L^2, I)}. \label{estimate nonlinear 1 scattering weighted space proof 2}
\end{align}
Similarly, by using the fractional chain rule, we have
\begin{align}
B_{21} \lesssim \|u\|^\alpha_{L^{m_2}_t(I, L^{q_2}_x)}\|\nabla u\|_{S(L^2, I)}. \label{estimate nonlinear 2 scattering weighted space proof 3}
\end{align}
We now estimate
\begin{align*}
B_{22} \leq \||x|^{-b-1} |u|^\alpha u \|_{L^{p_2'}_t(I, L^{q_2'}_x(B^c))} &\lesssim \||x|^{-b-1}\|_{L^{\gamma_2}_x(B^c)} \||u|^\alpha u\|_{L^{p_2'}_t(I, L^{\upsilon_2}_x)} \\
&\lesssim \|u\|^\alpha_{L^{m_2}_t(I, L^{q_2}_x)} \|u\|_{L^{p_2}_t(I, L^{n_2}_x)} \\
&\lesssim \|u\|^\alpha_{L^{m_2}_t(I, L^{q_2}_x)} \|\nabla u\|_{L^{p_2}_t(I, L^{q_2}_x)},
\end{align*}
provided that $(p_2, q_2)\in S$ and 
\begin{align*}
\frac{1}{q_2'}=\frac{1}{\gamma_2}+\frac{1}{\upsilon_2}, \quad \frac{d}{\gamma_2}<b+1, \quad \frac{1}{\upsilon_2} = \frac{\alpha}{q_2} + \frac{1}{n_2}, \quad \frac{1}{p_2'}=\frac{\alpha}{m_2}+\frac{1}{p_2}, \quad q_2<d, \quad \frac{1}{n_2} = \frac{1}{q_2}-\frac{1}{d}. 
\end{align*}
This is then equivalent to
\[
\frac{d}{\gamma_2}=d-\frac{d(\alpha+2)}{q_2}+1<b+1, \quad \frac{\alpha}{m_2}=1-\frac{2}{p_2}.
\]
Thus by choosing $q_2$ as in $(\ref{choice q_2})$, we obtain
\begin{align}
B_{22} \lesssim \|u\|^\alpha_{L^{m_2}_t(I, L^{q_1}_x)} \|\nabla u\|_{S(L^2, I)}. \label{estimate nonlinear 2 scattering weighted space proof 4}
\end{align}
Collecting $(\ref{estimate nonlinear 1 scattering weighted space proof 1}), (\ref{estimate nonlinear 1 scattering weighted space proof 2})$ and $(\ref{estimate nonlinear 2 scattering weighted space proof 1}), (\ref{estimate nonlinear 2 scattering weighted space proof 2}), (\ref{estimate nonlinear 2 scattering weighted space proof 3}), (\ref{estimate nonlinear 2 scattering weighted space proof 4})$, we obtain $(\ref{estimate nonlinear 1 scattering weighted space})$ and $(\ref{estimate nonlinear 2 scattering weighted space})$. \newline
\indent It remains to check that $p_1, p_2 <2\alpha+2$ where $(p_1, q_1), (p_2, q_2) \in S$ with $q_1, q_2$ as in $(\ref{choice q_1})$ and $(\ref{choice q_2})$ respectively. Note that $q_1, q_2$ are almost similar up to $\pm \eps$. Let us denote $(p,q) \in S$ with
\[
q=\frac{d(\alpha+2)}{d-b} + a \eps, \quad a \in \{\pm 1\}.
\]
We will check that for $\eps>0$ small enough, $p<2\alpha+2$ or $\frac{d}{2}-\frac{d}{q}=\frac{2}{p}>\frac{1}{\alpha+1}$. By a direct computation, it is equivalent to
\[
d[d\alpha^2 + (d-2+2b)\alpha +2b-4] +a\eps(d-b)[d(\alpha+1)-2]>0.
\]
Since $\alpha\geq \frac{4-2b}{d}>\alpha_0$ (see $(\ref{strauss exponent})$), we see that $d\alpha^2 + (d-2+2b)\alpha +2b-4>0$. Therefore, the above inequality holds true by taking $\eps>0$ sufficiently small.
\end{proof}

\begin{lemma} \label{lem scattering weighted space d=3}
Let $d=3$. Let 
\[
b\in \Big(0,\frac{5}{4}\Big), \quad \alpha \in \Big[\frac{4-2b}{3}, 3-2b \Big).
\]
Then there exist $(p_1, q_1), (p_2, q_2) \in S$ satisfying $2\alpha+2 >p_1, p_2$ and $q_1, q_2 \in (3, 6)$ such that
\begin{align}
\||x|^{-b} |u|^\alpha u\|_{S'(L^2, I)} &\lesssim \left(\|u\|^\alpha_{L^{m_1}_t(I, L^{q_1}_x)} + \|u\|^\alpha_{L^{m_2}_t(I, L^{q_2}_x)}\right) \|u\|_{S(L^2, I)}, \label{estimate 1 nonlinear scattering weighted space d=3} \\
\|\nabla(|x|^{-b} |u|^\alpha u)\|_{S'(L^2, I)} &\lesssim \left(\|u\|^\alpha_{L^{m_1}_t(I, L^{q_1}_x)} + \|u\|^\alpha_{L^{m_2}_t(I, L^{q_2}_x)}\right) \|\scal{\nabla} u\|_{S(L^2, I)}, \label{estimate 2 nonlinear scattering weighted space d=3} 
\end{align}
where $m_1 = \frac{\alpha p_1}{p_1-2}$ and $m_2 = \frac{\alpha p_2}{p_2-2}$. 
\end{lemma}
 
\begin{proof}
We firstly note that by using the same lines as in the proof of Lemma $\ref{lem scattering weighted space d geq 4}$, the following estimates
\begin{align}
\||x|^{-b} |u|^\alpha u \|_{S'(L^2, I)} &\lesssim \left(\|u\|^\alpha_{L^{m_1}_t(I, L^{q_1}_x)} + \|u\|^\alpha_{L^{m_2}_t(I, L^{q_2}_x)}\right) \|u\|_{S(L^2, I)},  \nonumber \\
\||x|^{-b} \nabla(|u|^\alpha u)\|_{S'(L^2, I)} &\lesssim \left(\|u\|^\alpha_{L^{m_1}_t(I, L^{q_1}_x)} + \|u\|^\alpha_{L^{m_2}_t(I, L^{q_2}_x)}\right) \|\nabla u\|_{S(L^2, I)} \label{estimate weighted solution d=3}
\end{align}
still hold true for $d=3, b\in (0,2)$ and $\alpha \in [\alpha_\star, \alpha^\star)$. It remains to estimate $\||x|^{-b-1} |u|^\alpha u\|_{S'(L^2, I)}$. To do this, we divide this term into two parts on $B$ and on $B^c$ which are denoted by $B_{12}$ and $B_{22}$ respectively. By H\"older's inequality and Remark $\ref{rem singularity bound}$,
\begin{align*}
B_{12}\leq \||x|^{-b-1}|u|^\alpha u\|_{L^{p_1'}_t(I, L^{q_1'}_x(B))} &\lesssim \||x|^{-b-1}\|_{L^{\gamma_1}_x(B)} \||u|^\alpha u\|_{L^{p_1'}_t(I, L^{\upsilon_1}_x)} \\
&\lesssim \|u\|^\alpha_{L^{m_1}_t(I, L^{q_1}_x)} \|u\|_{L^{p_1}_t(I, L^{n_1}_x)} \\
&\lesssim \|u\|^\alpha_{L^{m_1}_t(I, L^{q_1}_x)} \|\scal{\nabla}u\|_{L^{p_1}_t(I, L^{q_1}_x)},
\end{align*}
provided that $(p_1, q_1)\in S$ and 
\[
\frac{1}{q_1'}=\frac{1}{\gamma_1}+\frac{1}{\upsilon_1}, \quad \frac{d}{\gamma_1}>b+1, \quad\frac{1}{\upsilon_1}=\frac{\alpha}{q_1}+\frac{1}{n_1}, \quad \frac{1}{p_1'} = \frac{\alpha}{m_1}+\frac{1}{p_1},
\]
and
\[
q_1 \geq 3, \quad n_1 \in (q_1, \infty) \quad \text{or} \quad \frac{1}{n_1}=\frac{\tau}{q_1}, \quad \tau \in (0,1).
\]
This implies that 
\[
\frac{d}{\gamma_1}=3-\frac{3(\alpha+1+\tau)}{q_1}>b+1, \quad \frac{\alpha}{m_1}=1-\frac{2}{p_1}.
\]
Le us choose 
\begin{align}
q_1=\frac{3(\alpha+1+\tau)}{2-b} +\eps,\label{choice q_1 scattering weighted space}
\end{align}
for some $0<\eps\ll 1$ to be chosen later. Since $\alpha\geq \frac{4-2b}{3}$, it is obvious that $q_1>3$. Moreover, the condition $q_1 <6$ implies $\alpha+\tau<3-2b$. Thus by choosing $\tau$ closed to 0, we need $\alpha<3-2b$. Combining with $\alpha \geq \frac{4-2b}{3}$, we get
\begin{align}
\frac{4-2b}{3}\leq \alpha <3-2b, \quad 0<b<\frac{5}{4}. \label{requirement b alpha}
\end{align}
Thus, for $b$ and $\alpha$ satisfying $(\ref{requirement b alpha})$, we have
\[
B_{12} \lesssim \|u\|^\alpha_{L^{m_1}_t(I, L^{q_1}_x)} \|\scal{\nabla} u\|_{S(L^2, I)}.
\]
Similarly, we estimate
\begin{align*}
B_{22} \leq \||x|^{-b-1} |u|^\alpha u\|_{L^{p_2'}_t(I, L^{q_2'}_x(B^c))} &\lesssim \||x|^{-b-1}\|_{L^{\gamma_2}_x(B^c)} \||u|^\alpha u\|_{L^{p_2'}_t(I, L^{\upsilon_2}_x)} \\
&\lesssim \|u\|^\alpha_{L^{m_2}_t(I, L^{q_2}_x)} \|u\|_{L^{p_2}_t(I, L^{n_2}_x)} \\
&\lesssim \|u\|^\alpha_{L^{m_2}_t(I, L^{q_2}_x)} \|\scal{\nabla}u\|_{L^{p_2}_t(I, L^{q_2}_x)}
\end{align*}
provided that $(p_2, q_2)\in S$ and 
\[
\frac{1}{q_2'} = \frac{1}{\gamma_2} + \frac{1}{\upsilon_2}, \quad \frac{d}{\gamma_2}<b+1, \quad \frac{1}{\upsilon_2} = \frac{\alpha}{p_2} +\frac{1}{n_2}, \quad \frac{1}{p_2'} =\frac{\alpha}{m_2} +\frac{1}{p_2},
\]
and
\[
q_2\geq 3, \quad n_2 \in (q_2, \infty) \quad \text{or} \quad \frac{1}{n_2} =\frac{\tau}{q_2}, \quad \tau \in (0,1).
\]
We thus get
\[
\frac{d}{\gamma_2} = 3-\frac{3(\alpha+1+\tau)}{q_2} <b+1, \quad \frac{\alpha}{m_2} =1-\frac{2}{p_2}.
\]
Let us choose
\begin{align}
q_2=\frac{3(\alpha+1+\tau)}{2-b} -\eps, \label{choice q_2 scattering weighted space}
\end{align}
for some $0<\eps \ll 1$ to be chosen later. It is easy to see that $q_2 \in (3, 6)$ for $0<b<\frac{5}{4}, \frac{4-2b}{3}\leq \alpha<3-2b$ and $\eps>0$ small enough. We thus obtain
\[
B_{22} \lesssim \|u\|^\alpha_{L^{m_2}_t(I, L^{q_2}_x)} \|\scal{\nabla} u\|_{S(L^2, I)}.
\]
It remains to check $p_1, p_2<2\alpha+2$ for $(p_1,q_1), (p_2, q_2)\in S$ with $q_1$ and $q_2$ given in $(\ref{choice q_1 scattering weighted space})$ and $(\ref{choice q_2 scattering weighted space})$ respectively. Let us denote $(p,q)\in S$ with
\[
q=\frac{3(\alpha+1+\tau)}{2-b}+a\eps, \quad a \in \{\pm 1\}.
\]
The condition $p<2\alpha+2$ is equivalent to
\[
\frac{3}{2}-\frac{3}{q}=\frac{2}{p}>\frac{1}{\alpha+1}.
\]
A direct computation shows
\[
3[3\alpha^2 +2b\alpha +2b-3 + \tau(3\alpha +1)] + a\eps (2-b)(3\alpha+1) >0.
\]
By taking $\eps>0$ small enough and $\tau$ closed to 0, it is enough to have
\[
3\alpha^2 +2b\alpha +2b-3>0.
\]
It implies that $\alpha>\frac{3-2b}{3}$. Comparing with $(\ref{requirement b alpha})$, we see that
\[
\frac{4-2b}{3}\leq \alpha <3-2b, \quad b\in \Big(0, \frac{5}{4}\Big).
\]
The proof is complete.
\end{proof}

We also have the following result in the same spirit with Lemma $\ref{lem scattering weighted space d=3}$ in the two dimensional case. 
\begin{lemma} \label{lem scattering weighted space d=2}
Let $d=2$. Let $b\in (0,1)$ and $\alpha \in [\alpha_\star, \alpha^\star)$. Then there exist $(p_1, q_1), (p_2, q_2) \in S$ satisfying $2\alpha+2 >p_1, p_2$ and $q_1, q_2 \in (2, \infty)$ such that
\begin{align}
\||x|^{-b} |u|^\alpha u\|_{S'(L^2, I)} &\lesssim \left(\|u\|^\alpha_{L^{m_1}_t(I, L^{q_1}_x)} + \|u\|^\alpha_{L^{m_2}_t(I, L^{q_2}_x)}\right) \|u\|_{S(L^2, I)}, \label{estimate 1 nonlinear scattering weighted space d=2} \\
\|\nabla(|x|^{-b} |u|^\alpha u)\|_{S'(L^2, I)} &\lesssim \left(\|u\|^\alpha_{L^{m_1}_t(I, L^{q_1}_x)} + \|u\|^\alpha_{L^{m_2}_t(I, L^{q_2}_x)}\right) \|\scal{\nabla} u\|_{S(L^2, I)}, \label{estimate 2 nonlinear scattering weighted space d=2} 
\end{align}
where $m_1 = \frac{\alpha p_1}{p_1-2}$ and $m_2 = \frac{\alpha p_2}{p_2-2}$. 
\end{lemma} 

\begin{proof}
We firstly note that the following estimates
\begin{align}
\||x|^{-b} |u|^\alpha u \|_{S'(L^2, I)} &\lesssim \left(\|u\|^\alpha_{L^{m_1}_t(I, L^{q_1}_x)} + \|u\|^\alpha_{L^{m_2}_t(I, L^{q_2}_x)}\right) \|u\|_{S(L^2, I)},  \nonumber \\
\||x|^{-b} \nabla(|u|^\alpha u)\|_{S'(L^2, I)} &\lesssim \left(\|u\|^\alpha_{L^{m_1}_t(I, L^{q_1}_x)} + \|u\|^\alpha_{L^{m_2}_t(I, L^{q_2}_x)}\right) \|\nabla u\|_{S(L^2, I)} \label{estimate weighted solution d=2}
\end{align}
still hold true for $d=2, b\in (0,2)$ and $\alpha \in [\alpha_\star, \alpha^\star)$ by using the same lines as in the proof of Lemma $\ref{lem scattering weighted space d geq 4}$. It remains to estimate the term $\||x|^{-b-1} |u|^\alpha u\|_{S'(L^2, I)}$. Using the notations given in the proof of Lemma $\ref{lem scattering weighted space d geq 4}$, we bound this term by $B_{12}+B_{22}$. By H\"older's inequality and Remark $\ref{rem singularity bound}$,
\begin{align*}
B_{12} \leq \||x|^{-b-1} |u|^\alpha u\|_{L^{p_1'}_t(I, L^{q_1'}_x(B))} &\lesssim \||x|^{-b-1}\|_{L^{\gamma_1}_x(B)} \||u|^\alpha u\|_{L^{p_1'}_t (I, L^{\upsilon_1}_x)} \\
&\lesssim \|u\|^\alpha_{L^{m_1}_t(I, L^{q_1}_x)} \|u\|_{L^{p_1}_t (I, L^{n_1}_x)}\\
&\lesssim \|u\|^\alpha_{L^{m_1}_t(I, L^{q_1}_x)} \|\scal{\nabla}u\|_{L^{p_1}_t (I, L^{q_1}_x)},
\end{align*}
provided that $(p_1, q_1)\in S$ and
\[
\frac{1}{q_1'}=\frac{1}{\gamma_1}+\frac{1}{\upsilon_1}, \quad \frac{2}{\gamma_1}>b+1, \quad \frac{1}{\upsilon_1}=\frac{\alpha}{q_1}+\frac{1}{n_1}, 
\]
and
\[
q_1 \geq 2, \quad n_1 \in (q_1, \infty) \quad \text{or} \quad \frac{1}{n_1}=\frac{\tau}{q_1}, \quad \tau \in (0,1).
\]
These conditions imply that
\[
\frac{2}{\gamma_1}=2-\frac{2(\alpha+1+\tau)}{q_1}>b+1, \quad \frac{\alpha}{m_1}=1-\frac{2}{p_1}.
\]
Let us choose 
\begin{align}
q_1=\frac{2(\alpha+1+\tau)}{1-b} + \eps, \label{choice q_1 scattering weighted space d=2}
\end{align}
for some $0<\eps \ll 1$ to be chosen later. It is obvious that $q_1 \in (2,\infty)$ for any $\tau \in (0,1)$. We thus obtain
\[
B_{12}\lesssim  \|u\|^\alpha_{L^{m_1}_t(I, L^{q_1}_x)} \|\scal{\nabla}u\|_{S(L^2, I)}.
\]
Similarly, 
\begin{align*}
B_{22} \leq \||x|^{-b-1} |u|^\alpha u\|_{L^{p_2'}_t(I, L^{q_2'}_x(B^c))} &\lesssim \||x|^{-b-1}\|_{L^{\gamma_2}_x(B^c)} \||u|^\alpha u\|_{L^{p_2'}_t (I, L^{\upsilon_2}_x)} \\
&\lesssim \|u\|^\alpha_{L^{m_2}_t(I, L^{q_2}_x)} \|u\|_{L^{p_2}_t (I, L^{n_2}_x)}\\
&\lesssim \|u\|^\alpha_{L^{m_2}_t(I, L^{q_1}_x)} \|\scal{\nabla}u\|_{L^{p_2}_t (I, L^{q_2}_x)},
\end{align*}
provided that
\[
\frac{1}{q_2'}=\frac{1}{\gamma_2}+\frac{1}{\upsilon_2}, \quad \frac{2}{\gamma_1}<b+1, \quad \frac{1}{\upsilon_2}=\frac{\alpha}{q_2}+\frac{1}{n_2}, 
\]
and
\[
q_2 \geq 2, \quad n_2 \in (q_2, \infty) \quad \text{or} \quad \frac{1}{n_2}=\frac{\tau}{q_2}, \quad \tau \in (0,1).
\]
We learn from these conditions that
\[
\frac{d}{\gamma_2}=2-\frac{2(\alpha+1+\tau)}{q}<b+1, \quad \frac{\alpha}{m_2}=1-\frac{2}{p_2}.
\]
Let us choose
\begin{align}
q_2=\frac{2(\alpha+1+\tau)}{1-b}-\eps, \label{choice q_2 scattering weighted space d=2}
\end{align}
for some $0<\eps\ll 1$ small enough. By choosing $\eps>0$ sufficiently small, we have $q_2 \in (2, \infty)$ for any $\tau \in (0,1)$. We get
\[
B_{22} \lesssim \|u\|^\alpha_{L^{m_2}_t(I, L^{q_1}_x)} \|\scal{\nabla}u\|_{S(L^2, I)}.
\]
To complete the proof, we need to check $p_1, p_2 <2\alpha+2$ with $(p_1, q_1), (p_2, q_2)\in S$ where $q_1$ and $q_2$ given in $(\ref{choice q_1 scattering weighted space d=2})$ and $(\ref{choice q_2 scattering weighted space d=2})$ respectively. Let us denote $(p,q)\in S$ with 
\[
q=\frac{2(\alpha+1+\tau)}{1-b} +a \eps, \quad a\in \{\pm 1 \}.
\] 
The condition $p<2\alpha+2$ is equivalent to 
\[
1-\frac{2}{q}=\frac{2}{p}>\frac{1}{\alpha+1}.
\]
It is in turn equivalent to
\[
2[\alpha^2 +b\alpha + b-1 +\tau \alpha] + a\eps\alpha(1-b)>0.
\]
By taking $\eps>0$ small enough and $\tau$ closed to 0, this condition holds true provided $\alpha^2 +b\alpha +b-1>0$. This implies $\alpha>1-b$ which is satisfied since $\alpha \in [\alpha_\star, \alpha^\star)$. The proof is complete.
\end{proof}

As a direct consequence of Lemmas $\ref{lem scattering weighted space d geq 4}, \ref{lem scattering weighted space d=3}, \ref{lem scattering weighted space d=2}$, we have the following global $H^1$-Strichartz bound of solutions to the defocusing (INLS).
\begin{proposition} \label{prop global bound}
Let 
\[
d\geq 4, \quad b\in (0,2), \quad \alpha \in [\alpha_\star, \alpha^\star),
\]
or
\[
d=3, \quad b\in \Big(0, \frac{5}{4}\Big), \quad \alpha \in [\alpha_\star, 3-2b),
\]
or
\[
d=2, \quad b\in (0,1), \quad \alpha \in [\alpha_\star, \alpha^\star).
\]
Let $u_0 \in \Sigma$ and $u$ be the global solution to the defocusing \emph{(INLS)}. Then $u\in L^p(\R, W^{1,q}(\R^d))$ for any Schr\"odinger admissible pair $(p,q)$. 
\end{proposition}

\begin{proof}
We have from the Duhamel formula,
\begin{align}
u(t)=e^{it\Delta}u_0 -i\int_0^t e^{i(t-s)\Delta} |x|^{-b}|u(s)|^\alpha u(s)ds.\label{duhamel formula}
\end{align}
Let $0\leq T \leq t$. We apply Lemmas $\ref{lem scattering weighted space d geq 4}, \ref{lem scattering weighted space d=3}, \ref{lem scattering weighted space d=2}$ with $I=(T, t)$ and use the conservation of mass to get \footnote{See $(\ref{define h1 strichartz norm})$ for the definition of $\|u\|_{S(I)}$.}
\begin{align*}
\|u\|_{S(I)} &\leq  C\|u(T)\|_{H^1_x} + C \||x|^{-b}|u|^\alpha u\|_{S'(L^2, I)} + C\| \nabla(|x|^{-b}|u|^\alpha u)\|_{S'(L^2, I)} \\
& \leq C\|u_0\|_{H^1_x} + C\left(\|u\|^\alpha_{L^{m_1}_t(I, L^{q_1}_x)} + \|u\|^\alpha_{L^{m_2}_t(I, L^{q_2}_x)}\right) \|u\|_{S(I)},
\end{align*}
where $(p_i, q_i)\in S$ satisfy $p_i<2\alpha+2, q_i \in (2, 2^\star)$ and $m_i =\frac{\alpha p_i}{p_i-2}$ for $i=1,2$. Here $2^\star= \frac{2d}{d-2}$ if $d\geq 3$ and $2^\star = \infty$ if $d=2$. Note that the constant $C$ is independent of $I$ and may change from line to line. The norm $\|u\|^\alpha_{L^{m_i}_t(I, L^{q_i}_x)}$ can be written as
\[
\Big(\int_T^t \|u(s)\|^{m_i}_{L^{q_i}_x} ds \Big)^{\frac{\alpha}{m_i}} = \Big(\int_T^t \|u(s)\|^{\frac{\alpha p_i}{p_i-2}}_{L^{q_i}_x} ds \Big)^{\frac{p_i-2}{p_i}}. 
\]
By the decay of global solutions given in Theorem $\ref{theorem decay property}$, we see that
\[
\|u(s)\|_{L^{q_i}_x} \lesssim s^{-d\left(\frac{1}{2}-\frac{1}{q_i}\right)} =s^{-\frac{2}{p_i}} \quad \text{so} \quad \|u(s)\|^{\frac{\alpha p_i}{p_i-2}}_{L^{q_i}_x} \lesssim s^{-\frac{2\alpha}{p_i-2}}. 
\]
Since $p_i <2\alpha +2$ or $\frac{2\alpha}{p_i-2}>1$, by choosing $T>0$ large enough,
\[
C\Big(\int_T^t \|u(s)\|^{\frac{\alpha p_i}{p_i-2}}_{L^{q_i}_x} ds \Big)^{\frac{p_i-2}{p_i}} \leq \frac{1}{4}.
\]
We thus obtain
\[
\|u\|_{S(I)} \leq C +\frac{1}{2} \|u\|_{S(I)} \quad \text{or}\quad \|u\|_{S(I)} \leq 2C.
\]
Letting $t\rightarrow +\infty$, we obtain $\|u\|_{S((T, +\infty))} \leq 2C$. Similarly, one can prove that $\|u\|_{S((-\infty,-T))} \leq 2C$. Combining these two bounds and the local theory, we prove $u \in L^p(\R, W^{1,q}(\R^d))$ for any Schr\"odinger admissible pair $(p,q)$. 
\end{proof}

\begin{remark} \label{rem global strichartz bound}
Using this global $H^1$-Strichartz bound, one can obtain easily (see the proof of Theorem $\ref{theorem scattering weighted space}$ given below) the scattering in $H^1$ provided that $u_0 \in \Sigma$. But one does not know whether the scattering states $u_0^\pm$ belong to $\Sigma$. 
\end{remark}

In order to show the scattering states $u_0^\pm \in \Sigma$, we need to show the global $L^2$-Strichartz bound for the weighted solutions $(x+2it\nabla)u(t)$. To do this, we need the following estimates on the nonlinearity.
\begin{lemma} \label{lem nonlinear estimate weighted solution}
1. Let
\[
d=3, \quad b\in (0,1), \quad \alpha \in \Big(\frac{5-2b}{3}, 3-2b\Big).
\]
Then there exist $(p_1, q_1), (p_2, q_2)\in S$ satisfying $\alpha+1>p_1, p_2$ and $q_1, q_2 \in  (3,6)$ such that
\[
\||x|^{-b-1}|u|^\alpha u\|_{S'(L^2, I)} \lesssim \left(\|u\|^\alpha_{L^{m_1}_t(I, L^{q_1}_x)} +\|u\|^\alpha_{L^{m_2}_t(I, L^{q_2}_x)}\right) \|\scal{\nabla}u\|_{S(L^2, I)}.
\]
2. Let
\[
d=2, \quad b\in (0,1), \quad \alpha \in [\alpha_\star, \alpha^\star).
\]
Then there exist $(p_1, q_1), (p_2, q_2)\in S$ satisfying $\alpha+1>p_1, p_2$ and $q_1, q_2 \in  (2,\infty)$ such that
\[
\||x|^{-b-1}|u|^\alpha u\|_{S'(L^2, I)} \lesssim \left(\|u\|^\alpha_{L^{m_1}_t(I, L^{q_1}_x)} +\|u\|^\alpha_{L^{m_2}_t(I, L^{q_2}_x)}\right) \|\scal{\nabla}u\|_{S(L^2, I)}.
\]
\end{lemma}

\begin{proof}
In the case $d=3$, we use the same argument as in the proof of Lemma $\ref{lem scattering weighted space d=3}$ with 
\[
q_1 = \frac{3(\alpha+1+\tau)}{2-b} +\eps, \quad q_2 = \frac{3(\alpha+1+\tau)}{2-b} -\eps
\]
for some $\eps>0$ small enough and $\tau$ closed to 0. It remains to check $\alpha+1>p_1, p_2$ where $(p_1, q_1), (p_2, q_2) \in S$. Let us denote $(p,q) \in S$ with 
\[
q=\frac{3(\alpha+1+\tau)}{2-b} + a\eps, \quad a \in \{\pm 1\}.
\] 
The condition $p<\alpha+1$ is equivalent to
\[
\frac{3}{2}-\frac{3}{q} =\frac{2}{p} >\frac{2}{\alpha+1}.
\]
An easy computation shows
\[
3[3\alpha^2 + 2(b-1)\alpha + 2b-5 + \tau(3\alpha-1)] +a\eps (2-b) (3\alpha-1)>0.
\]
By taking $\eps$ and $\tau$ small enough, it is enough to show
\[
3\alpha^2 + 2(b-1)\alpha+ 2b-5>0. 
\]
This implies that $\alpha >\frac{5-2b}{3}$. Comparing with the assumptions $b\in \left(0, \frac{5}{4}\right)$ and $\alpha \in \left[\frac{4-2b}{3}, 3-2b\right)$ of Lemma $\ref{lem scattering weighted space d=3}$, we have
\[
b\in (0,1), \quad \alpha \in \Big(\frac{5-2b}{3}, 3-2b \Big).
\]
The case $d=2$ is treated similarly. As in the proof of Lemma $\ref{lem scattering weighted space d=2}$, we choose 
\[
q_1 = \frac{2(\alpha+1+\tau)}{1-b} +\eps, \quad q_2 =\frac{2(\alpha+1+\tau)}{1-b} -\eps,
\]
for some $\eps, \tau>0$ small enough. As above, let us denote $(p,q) \in S$ with 
\[
q=\frac{2(\alpha+1+\tau)}{1-b} + a\eps, \quad a \in \{\pm 1\}.
\] 
The condition $p<\alpha+1$ is equivalent to
\[
1-\frac{2}{q} =\frac{2}{p} >\frac{2}{\alpha+1}.
\]
An easy computation shows
\[
2[\alpha^2 + (b-1)\alpha + b-2 + \tau(\alpha-1)] +a\eps (1-b) (\alpha-1)>0.
\]
By taking $\eps$ and $\tau$ small enough, it is enough to show
\[
\alpha^2 + (b-1)\alpha+ b-2>0. 
\]
This implies that $\alpha >1-b$ which is always satisfied for $\alpha \in [\alpha_\star, \alpha^\star)$. The proof is complete.
\end{proof}

\begin{proposition} \label{prop global bound weighted space}
Let $d, b$ and $\alpha$ be as in Theorem $\ref{theorem scattering weighted space}$. Let $u_0 \in \Sigma$ and $u$ be the global solution to the defocusing \emph{(INLS)}. Set
\[
w(t) := (x+2it\nabla) u(t).
\] 
Then $w\in L^p(\R, L^q(\R^d))$ for every Schr\"odinger admissible pair $(p,q)$. 
\end{proposition}

\begin{proof}
We firstly notice that $x+2it\nabla$ commutes with $i\partial_t +\Delta$. By Duhamel's formula,
\begin{align}
w(t)=e^{it\Delta} x u_0 -i\int_0^t e^{i(t-s)\Delta} (x+2is\nabla)(|x|^{-b} |u(s)|^\alpha u(s)) ds. \label{duhamel formula weighted}
\end{align}
Let $v$ be as in $(\ref{define v})$. By $(\ref{change variable property 1})$, we have
\[
|(x+2it\nabla)(|x|^{-b}|u|^\alpha u)| = 2|t||\nabla(|x|^{-b}|v|^\alpha v)|, \quad |v|=|u|, \quad 2|t||\nabla v| = |w|.
\] 
\indent \underline{\bf Case 1: $d\geq 4$.} Strichartz estimates and Lemma $\ref{lem scattering weighted space d geq 4}$ show that for any $t>0$ and $I=(0,t)$,
\begin{align*}
\|w\|_{S(L^2, I)} &\lesssim \|x u_0\|_{L^2_x} + \|(x+2is\nabla)(|x|^{-b}|u|^\alpha u)\|_{S'(L^2,I)} \\
&\lesssim \|x u_0\|_{L^2_x}+  \|2|s| \nabla(|x|^{-b}|v|^\alpha v)\|_{S'(L^2, I)}. 
\end{align*}
Let $0\leq T \leq t$. We bound
\[
\|2|s| \nabla(|x|^{-b}|v|^\alpha v)\|_{S'(L^2, I)} \leq \|2|s| \nabla(|x|^{-b}|v|^\alpha v)\|_{S'(L^2, (0,T))} + \|2|s| \nabla(|x|^{-b}|v|^\alpha v)\|_{S'(L^2, (T,t))} = A+B.
\]
The term $A$ is treated as follows. By Lemma $\ref{lem scattering weighted space d geq 4}$ and keeping in mind that $|v|=|u|, 2|s||\nabla v| = |w|$, we bound
\begin{align*}
A &\lesssim \Big(\|u\|^\alpha_{L^{m_1}_t((0,T), L^{q_1}_x)} + \|u\|^\alpha_{L^{m_2}_t((0,T), L^{q_2}_x)}\Big) \|2|s| \nabla v\|_{S(L^2, I)} \\
&\lesssim \Big(\|u\|^\alpha_{L^{m_1}_t((0,T), L^{q_1}_x)} + \|u\|^\alpha_{L^{m_2}_t((0,T), L^{q_2}_x)}\Big) \|w\|_{S(L^2, I)},
\end{align*}
for some $(p_i, q_i)\in S$ satisfy $p_i<2\alpha+2, q_i \in (2, 2^\star)$ and $m_i =\frac{\alpha p_i}{p_i-2}$ for $i=1,2$. We next estimate
\[
\|u\|^\alpha_{L^{m_i}_t((0,T), L^{q_i}_x)} \lesssim T^{\frac{\alpha}{m_i}} \|u\|^\alpha_{L^\infty_t((0,T), H^1_x)} <\infty, \quad i =1, 2.
\]
Here the time $T>0$ is large but fixed and $u\in L^\infty_t((0,T), H^1_x)$ by the local theory. We also have $\|w\|_{S(L^2,(0,T))}<\infty$ which is proved in the Appendix. This shows the boundedness of $A$. For the term $B$, we bound
\begin{align*}
B &\lesssim \Big(\|u\|^\alpha_{L^{m_1}_t((T,t), L^{q_1}_x)} + \|u\|^\alpha_{L^{m_2}_t((T,t), L^{q_2}_x)}\Big) \|2|s| \nabla v\|_{S(L^2, (T, t))} \\
&\lesssim \Big(\|u\|^\alpha_{L^{m_1}_t((T,t), L^{q_1}_x)} + \|u\|^\alpha_{L^{m_2}_t((T,t), L^{q_2}_x)}\Big) \|w\|_{S(L^2, I)},
\end{align*}
for some $(p_i, q_i)\in S$ satisfy $p_i<2\alpha+2, q_i \in (2, 2^\star)$ and $m_i =\frac{\alpha p_i}{p_i-2}$ for $i=1,2$. By the same argument as in the proof of Proposition $\ref{prop global bound}$, we see that $\|u\|^\alpha_{L^{m_i}_t(T, t)}$ is small for $T>0$ large enough. Therefore,
\[
\|w\|_{S(L^2, I)} \leq C + \frac{1}{2}\|w\|_{S(L^2, I)} \quad \text{or} \quad \|w\|_{S(L^2, I)} \leq 2C.
\]
Letting $t\rightarrow +\infty$, we prove that $\|w\|_{S(L^2, (0,+\infty))} \leq 2C$. Similarly, one proves as well that $\|w\|_{S(L^2, (-\infty,0))} \leq 2C$. This shows $w\in L^p(\R, L^q(\R^d))$ for any Schr\"odinger admissible pair $(p,q)$. \newline
\indent \underline{\bf Case 2: $d=2,3$.} We bound
\begin{align*}
\|w\|_{S(L^2,I)}&\lesssim \|xu_0\|_{L^2_x} + \|(x+2is\nabla)(|x|^{-b} |u|^\alpha u)\|_{S'(L^2, I)}  \\
&\lesssim \|xu_0\|_{L^2_x} + \|2|s|\nabla(|x|^{-b} |v|^\alpha v)\|_{S'(L^2, I)}  \\
&\lesssim \|xu_0\|_{L^2_x} + \|2|s||x|^{-b} \nabla(|v|^\alpha v)\|_{S'(L^2, I)} + \|2|s||x|^{-b-1} |v|^\alpha v\|_{S'(L^2, I)} \\
&\lesssim \|x u_0\|_{L^2_x} + A + B.
\end{align*}
The term $A$ is treated similarly as in Case 1 using $(\ref{estimate weighted solution d=3}), (\ref{estimate weighted solution d=2})$. 
It remains to bound the term $B$. By Lemma $\ref{lem nonlinear estimate weighted solution}$,
\begin{align*}
B 
\lesssim \left(\||s|^{\frac{1}{\alpha}} u\|^\alpha_{L^{m_1}_t(I, L^{q_1}_x)} + \||s|^{\frac{1}{\alpha}} u\|^\alpha_{L^{m_2}_t(I, L^{q_2}_x)} \right) \|u\|_{S(L^2,I)},
\end{align*}
for some $(p_i, q_i)\in S$ satisfy $p_i<\alpha+1, q_i \in (2, 2^\star)$ and $m_i =\frac{\alpha p_i}{p_i-2}$ for $i=1,2$. We learn from Proposition $\ref{prop global bound}$ that $\|u\|_{S(L^2, I)}<\infty$. Let us bound $\||s|^{\frac{1}{\alpha}} u\|^\alpha_{L^{m_i}_t(I, L^{q_i}_x)}$ for $i=1,2$. To do so, we split $I$ into $(0,T)$ and $(T,t)$. By Sobolev embedding 
\[
\||s|^{\frac{1}{\alpha}} u\|^\alpha_{L^{m_i}_t((0,T), L^{q_i}_x)} \lesssim T^{1+\frac{\alpha}{m_i}}\|u\|^\alpha_{L^\infty_t((0,T), H^1_x)} <\infty.
\] 
We next write
\[
\||s|^{\frac{1}{\alpha}} u\|^\alpha_{L^{m_i}_t((T,t), L^{q_i}_x)} = \Big(\int_T^t |s|^{\frac{m_i}{\alpha}}\|u(s)\|^{m_i}_{L^{q_i}_x} ds \Big)^{\frac{\alpha}{m_i}}.
\]
By the decay of global solutions given in Theorem $\ref{theorem decay property}$, we see that
\[
|s|^{\frac{m_i}{\alpha}}\|u(s)\|^{m_i}_{L^{q_i}_x} \lesssim |s|^{\frac{m_i}{\alpha}-m_i\left(\frac{d}{2}-\frac{d}{q_i} \right)} = |s|^{-m_i\left(\frac{2}{p_i}-\frac{1}{\alpha}\right)} = |s|^{-\frac{2\alpha-p_i}{p_i-2}}.
\]
Since $p_i<\alpha+1$ or $\frac{2\alpha-p_i}{p_i-2}>1$, by taking $T>0$ sufficiently large, we see that $\||s|^{\frac{1}{\alpha}} u\|^\alpha_{L^{m_i}_t((T,t), L^{q_i}_x)}$ is small. This proves that the term $B$ is bounded for some $T>0$ large enough. Therefore,
\[
\|w\|_{S(L^2, I)} \leq C +\frac{1}{2}\|w\|_{S(L^2, I)} \quad \text{or} \quad \|w\|_{S(L^2, I)} \leq 2C.
\]
By letting $t$ tends to $+\infty$, we complete the proof.
\end{proof}
We are now able to prove Theorem $\ref{theorem scattering weighted space}$. The proof follows by a standard argument (see e.g. \cite{Cazenave} or \cite{Tao}).

\paragraph{\bf Proof of Theorem \ref{theorem scattering weighted space}.}
Let $u$ be the global solution to the defocusing (INLS). By the time reserval symmetry, we only consider the positive time. The Duhamel formula $(\ref{duhamel formula})$ implies
\[
e^{-it\Delta} u(t)= u_0  -i \int_0^t e^{-is\Delta} |x|^{-b}|u(s)|^\alpha u(s)ds. 
\]
Let $0<t_1<t_2<\infty$. By Strichartz estimates and Lemmas $\ref{lem scattering weighted space d geq 4}, \ref{lem scattering weighted space d=3}, \ref{lem scattering weighted space d=2}$,
\begin{align*}
\|e^{-it_2 \Delta} u(t_2) - e^{-it_1 \Delta} u(t_1) \|_{H^1_x} &= \Big\|\int_{t_1}^{t_2} e^{-is \Delta} |x|^{-b} |u(s)|^\alpha u(s) ds \Big\|_{H^1_x} \\
&\lesssim \||x|^{-b} |u|^\alpha u\|_{S'(L^2, (t_1, t_2))} + \|\nabla(|x|^{-b} |u|^\alpha u)\|_{S'(L^2, (t_1, t_2))} \\
&\lesssim \Big(\|u\|^\alpha_{L^{m_1}_t((t_1, t_2), L^{q_1}_x)} + \|u\|^\alpha_{L^{m_2}_t((t_1, t_2), L^{q_2}_x)}\Big) \|u\|_{S((t_1, t_2))}, 
\end{align*}
where $(p_i, q_i)\in S$ satisfy $p_i<2\alpha+2, q_i \in (2, 2^\star)$ and $m_i =\frac{\alpha p_i}{p_i -2}$ for $i=1,2$. By the same argument as in Proposition $\ref{prop global bound}$ and the global bound $\|u\|_{S(\R)}<\infty$, we see that
\[
\Big(\|u\|^\alpha_{L^{m_1}_t((t_1, t_2), L^{q_1}_x)} + \|u\|^\alpha_{L^{m_2}_t((t_1, t_2), L^{q_2}_x)}\Big) \|u\|_{S((t_1, t_2))} \rightarrow 0,
\]
as $t_1, t_2 \rightarrow +\infty$. This shows that $e^{-it\Delta} u(t)$ is a Cauchy sequence in $H^1(\R^d)$ as $t\rightarrow +\infty$. Therefore, there exists $u_0^+ \in H^1(\R^d)$ such that $e^{-it\Delta} u(t) \rightarrow u_0^+$ as $t\rightarrow +\infty$. Note that this convergence holds for $d, b$ and $\alpha$ as in Proposition $\ref{prop global bound}$. We now show that this scattering state $u_0^+$ belongs to $\Sigma$. To do so, we firstly observe that the operator $x+2it\nabla$ can be written as
\begin{align}
x+2it\nabla = e^{it\Delta} x e^{-it\Delta}. \label{pseudo conformal operator property}
\end{align}
Indeed, since $x+2it\nabla$ commutes with $i\partial_t +\Delta$, we see that if $u$ is a solution to the linear Schr\"odinger equation, then so is $(x+2it\nabla) u$. Thus, if we set $u(t)=e^{it\Delta} \varphi$, then 
\[
(x+2it\nabla) u(t) = e^{it\Delta} x \varphi.
\]
By setting $\varphi = e^{-it\Delta} \psi$, we see that
\[
(x+2it\nabla) \psi = e^{it\Delta} x e^{-it\Delta} \psi,
\]
which proves $(\ref{pseudo conformal operator property})$. Using the Duhamel formula $(\ref{duhamel formula weighted})$ and $(\ref{pseudo conformal operator property})$, we have
\[
x e^{-it\Delta} u(t) = xu_0 -i\int_0^t e^{-is\Delta} (x+2is\nabla) (|x|^{-b}|u(s)|^\alpha u(s)) ds.
\]
\indent \underline{\bf Case 1: $d\geq 4$.} By Strichartz estimates, Lemma $\ref{lem scattering weighted space d geq 4}$ and using the same argument as in Proposition $\ref{prop global bound weighted space}$, we see that
\begin{align*}
\|xe^{-t_2 \Delta} u(t_2) - xe^{-it_1 \Delta} u(t_1) \|_{L^2_x} &= \Big\|\int_{t_1}^{t_2} e^{-its \Delta} (x+2is\nabla)(|x|^{-b} |u(s)|^\alpha u(s)) ds \Big\|_{L^2_x} \\
&\lesssim \|(x+2is\nabla) (|x|^{-b}|u|^\alpha u)\|_{S'(L^2, (t_1, t_2))} \\
&\lesssim \|2|s| \nabla(|x|^{-b}|v|^\alpha v)\|_{S'(L^2, (t_1, t_2))} \\
&\lesssim \left(\|u\|^\alpha_{L^{m_1}_t((t_1, t_2), L^{q_1}_x)} + \|u\|^\alpha_{L^{m_2}_t((t_1, t_2), L^{q_2}_x)} \right) \|2|s| \nabla v\|_{S(L^2, (t_1, t_2))}  \\
&\lesssim \left(\|u\|^\alpha_{L^{m_1}_t((t_1, t_2), L^{q_1}_x)} + \|u\|^\alpha_{L^{m_2}_t((t_1, t_2), L^{q_2}_x)} \right) \|w\|_{S(L^2, (t_1, t_2))}, 
\end{align*}
where $(p_i, q_i)\in S$ satisfy $p_i<2\alpha+2, q_i \in (2, 2^\star)$ and $m_i =\frac{\alpha p_i}{p_i -2}$ for $i=1,2$. Arguing as in the proof of Proposition $\ref{prop global bound weighted space}$ and the global bound $\|w\|_{S(L^2, \R)}<\infty$, we see that
\[
\left(\|u\|^\alpha_{L^{m_1}_t((t_1, t_2), L^{q_1}_x)} + \|u\|^\alpha_{L^{m_2}_t((t_1, t_2), L^{q_2}_x)} \right) \|w\|_{S(L^2, (t_1, t_2))} \rightarrow 0,
\]
as $t_1, t_2 \rightarrow +\infty$. \newline
\indent \underline{\bf Case 2: $d=2,3$.} 
\begin{align*}
\|xe^{-t_2 \Delta} u(t_2) - xe^{-it_1 \Delta} u(t_1) \|_{L^2_x} &= \Big\|\int_{t_1}^{t_2} e^{-its \Delta} (x+2is\nabla)(|x|^{-b} |u(s)|^\alpha u(s)) ds \Big\|_{L^2_x} \\
&\lesssim \|(x+2is\nabla) (|x|^{-b}|u|^\alpha u)\|_{S'(L^2, (t_1, t_2))} \\
&\lesssim \|2|s| \nabla(|x|^{-b}|v|^\alpha v)\|_{S'(L^2, (t_1, t_2))} \\
&\lesssim \|2|s||x|^{-b} \nabla (|v|^\alpha v)\|_{S'(L^2, (t_1, t_2))} + \|2|s| |x|^{-b-1} |v|^\alpha v\|_{S'(L^2, (t_1, t_2))} \\
&=:A+B.
\end{align*}
For term $A$, we use  $(\ref{estimate weighted solution d=3}), (\ref{estimate weighted solution d=2})$ and the fact $|v|=|u|, 2|s||\nabla v| = |w|$ to have
\begin{align}
A &\lesssim \left(\| u\|^\alpha_{L^{m_1}_t((t_1, t_2), L^{q_1}_x)} + \|u\|^\alpha_{L^{m_2}_t((t_1, t_2), L^{q_2}_x)} \right) \|2|s|\nabla v\|_{S(L^2,(t_1, t_2))} \nonumber \\
&\lesssim \left(\| u\|^\alpha_{L^{m_1}_t((t_1, t_2), L^{q_1}_x)} + \|u\|^\alpha_{L^{m_2}_t((t_1, t_2), L^{q_2}_x)} \right) \|w\|_{S(L^2,(t_1, t_2))}, \label{term A}
\end{align}
for some $(p_i, q_i)\in S$ satisfy $p_i<2\alpha+2, q_i \in (2, 2^\star)$ and $m_i =\frac{\alpha p_i}{p_i-2}$ for $i=1,2$. Similarly, by Lemma $\ref{lem nonlinear estimate weighted solution}$,
\begin{align}
B 
\lesssim \left(\||s|^{\frac{1}{\alpha}} u\|^\alpha_{L^{m_1}_t((t_1, t_2), L^{q_1}_x)} + \||s|^{\frac{1}{\alpha}} u\|^\alpha_{L^{m_2}_t((t_1, t_2), L^{q_2}_x)} \right) \|u\|_{S(L^2,(t_1, t_2))}, \label{term B}
\end{align}
for some $(p_i, q_i)\in S$ satisfy $p_i<\alpha+1, q_i \in (2, 2^\star)$ and $m_i =\frac{\alpha p_i}{p_i-2}$ for $i=1,2$. By the same argument as in Case 2 of the proof of Proposition $\ref{prop global bound weighted space}$, we see that the right hand sides of $(\ref{term A})$ and $(\ref{term B})$ tend to 0 as $t_1, t_2 \rightarrow +\infty$. \newline
\indent In both cases, we show that $x e^{-it\Delta} u(t)$ is a Cauchy sequence in $L^2$ as $t\rightarrow +\infty$. We thus have $x u_0^+ \in L^2$ and so $u_0^+ \in \Sigma$. Moreover,
\[
u_0^+= u_0 -i \int_t^\infty e^{-is \Delta} |x|^{-b}|u(s)|^\alpha u(s)ds.
\] 
\indent By repeating the above estimates, we prove as well that
\[
\|e^{-it\Delta} u(t)-u_0^+\|_{\Sigma} \rightarrow 0,
\]
as $t\rightarrow +\infty$. The proof is complete.
\hfill $\Box$

\begin{remark}
We end this section by giving some comments on the scattering in $\Sigma$ for $\alpha \in (0,\alpha_\star)$. In this case, by Theorem $\ref{theorem decay property}$, we have the following decay of global solutions to the defocusing (INLS)
\begin{align}
\|u(t)\|_{L^q_x} \lesssim |t|^{-\frac{d(2b+d\alpha)}{4}\left(\frac{1}{2}-\frac{1}{q}\right)}, \label{decay low range}
\end{align}
for $q$ as in $(\ref{define q})$. Let us consider the easiest case $d\geq 4$. In order to obtain the global $H^1$-Strichartz bound on $u$ and the global $L^2$-Strichartz bound on $w$ (see Proposition $\ref{prop global bound}$ and Proposition $\ref{prop global bound weighted space}$), we need $\|u\|^\alpha_{L^m_t((T, t), L^q_x)}$ to be small as $T>0$ large enough, where $(p,q) \in S$ and $m=\frac{\alpha p}{p-2}$. This norm can be written as
\begin{align}
 \Big(\int_T^t \|u(s)\|_{L^q_x}^m ds \Big)^{\frac{\alpha}{m}} = \Big( \int_T^t \|u(s)\|_{L^q_x}^{\frac{\alpha p}{p-2}} ds \Big)^{\frac{p-2}{p}}. \label{norm low range}
\end{align}
Using $(\ref{decay low range})$, 
\[
\|u(s)\|_{L^q_x}^{\frac{\alpha p}{p-2}} \lesssim s^{-\frac{\alpha(2b+d\alpha)}{2(p-2)}}.
\]
To make the right hand side of $(\ref{norm low range})$ small, we need $\frac{\alpha(2b+d\alpha)}{2(p-2)}>1$ or equivalently $2p<4+\alpha(2b+d\alpha)$ hence
\begin{align}
\frac{d}{2}-\frac{d}{q} =\frac{2}{p}>\frac{4}{4+\alpha(2b+d\alpha)}. \label{condition low range}
\end{align}
Let us choose $q$ as in the proof of Lemma $\ref{lem scattering weighted space d geq 4}$, i.e.
\[
q=\frac{d(\alpha+2)}{d-b}+a\eps, \quad a\in \{\pm 1\},
\]
for some $\eps>0$ small enough. We see that $(\ref{condition low range})$ is equivalent to
\[
d[d^2 \alpha^3 + 4bd \alpha^2 + (4d-8+4b^2)\alpha + 8b-16] + a\eps(d-b)[4d-8+ d\alpha(2b+d\alpha)]>0.
\]
By taking $\eps>0$ small enough, it is enough to show $f(\alpha):=d^2 \alpha^3 + 4bd \alpha^2 + (4d-8+4b^2)\alpha + 8b-16>0$. Since $b\in (0,2)$, we see that $f(0)=8b-16<0$ and $f(\alpha_\star)=f\left(\frac{4-2b}{d}\right) = \frac{8(4-2b)}{d}>0$. Hence $f(\alpha)=0$ has a solution in $(0,\alpha_\star)$. Thus, the inequality $f(\alpha)>0$ holds true for a sub interval of $(0,\alpha_\star)$. By the same argument as for the case $\alpha \in [\alpha_\star, \alpha^\star)$, we can obtain a similar scattering result in $\Sigma$ for a certain range of $\alpha \in (0,\alpha_\star)$. 
\end{remark}

\appendix
\section{Local $L^2$-Strichartz bound of weighted solutions}
\begin{lemma} \label{lem local bound weighted space}
Let $d, b$ and $\alpha$ be as in Theorem $\ref{theorem local existence}$. Let $u_0 \in \Sigma$ and $u$ be the corresponding global solutions to the defocusing \emph{(INLS)}. Set 
\[
w(t)=(x+2it\nabla)u(t).
\]
Then $w \in L^p_{\emph{loc}}(\R, L^q(\R^d))$ for any Schr\"odinger admissible pair $(p,q)$.
\end{lemma}
\begin{proof}
We follow the argument of Tao, Visan, and Zhang \cite{TVZ}. For simplifying the notation, we denote $H(t)=x+2it\nabla$. We will show that $\|Hu\|_{S(L^2,I)} <\infty$ for any finite time interval $I$ of $\R$. By the time reversal symmetry, we may assume $I=[0,T]$.  We split $I$ into a finite number of subintervals $I_j=[t_j, t_{j+1}]$ such that $|I_j| <\eps$ for some small constant $\eps>0$ to be chosen later. \newline
\indent \underline{\bf Case 1: $d\geq 4, b\in (0,2)$ or $d=3, b\in (0,1)$ and $\alpha \in (0, \alpha^\star)$.} By $(\ref{pseudo conformal operator property})$, we see that on each interval $I_j$,
\[
H(t) u(t) = e^{i(t-t_j)\Delta} H(t_j) u(t_j) - i\int_{t_j}^t e^{i(t-s)\Delta} H(s)(|x|^{-b}|u(s)|^\alpha u(s)) ds.
\]
Let $v$ be as in $(\ref{define v})$. By Strichartz estimates and \eqref{non-est-2-d-geq3} and that $|v(s)|=|u(s)|, 2|s||\nabla v(s)| = |H(s)u(s)|$, we have
\begin{align*}
\|Hu\|_{S(L^2, I_j)} &\lesssim \|H(t_j) u(t_j)\|_{L^2_x} + \|H(s)(|x|^{-b}|u|^\alpha u)\|_{S'(L^2, I_j)} \\
&\lesssim \|H(t_j) u(t_j)\|_{L^2_x}+ \|2|s| \nabla (|x|^{-b} |v|^\alpha v)\|_{S'(L^2,I_j)} \\
&\lesssim \|H(t_j) u(t_j)\|_{L^2_x}+ \left(|I_j|^{\theta_1} + |I_j|^{\theta_2} \right) \|\nabla u\|^\alpha_{S(L^2, I_j)} \|2|s|\nabla v\|_{S(L^2, I_j)} \\
&\lesssim \|H(t_j) u(t_j)\|_{L^2_x}+ \left(\eps^{\theta_1} + \eps^{\theta_2} \right) \|u\|^\alpha_{S(I_j)} \|Hu\|_{S(L^2, I_j)}.
\end{align*}
Since $\|u\|_{S(\R)}<\infty$, by choosing $\eps>0$ small enough depending on $T, \|u\|_{S(\R)}$, we get
\[
\|Hu\|_{S(L^2, I_j)} \lesssim \|H(t_j)u(t_j)\|_{L^2_x}.
\]
By induction, we have for each $j$,
\[
\|Hu\|_{S(L^2, I_j)} \lesssim \|H(0)u(0)\|_{L^2_x}=\|xu_0\|_{L^2_x}.
\]
Summing these estimates over all subintervals $I_j$, we obtain
\[
\|Hu\|_{S(L^2, I)} <\infty. 
\]
\indent \underline{\bf Case 2: $d=3, b\in \left[1,\frac{3}{2}\right)$ and $\alpha \in \left(0, \frac{6-4b}{2b-1}\right)$ or $d=2, b\in (0,1)$ and $\alpha \in (0, \alpha^\star)$.} By Strichartz estimates, \eqref{non-est-2-d3}, \eqref{non-est-2-d2} and keeping in mind that $|v|=|u|, 2|s||\nabla v| = |Hu|$, we bound
\begin{align*}
\|Hu\|_{S(L^2, I_j)} &\lesssim \|H(t_j) u(t_j)\|_{L^2_x} + \|H(s)(|x|^{-b}|u|^\alpha u)\|_{S'(L^2, I_j)} \\
&\lesssim \|H(t_j) u(t_j)\|_{L^2_x}+ \|2|s| \nabla (|x|^{-b} |v|^\alpha v)\|_{S'(L^2,I_j)} \\
&\lesssim \|H(t_j) u(t_j)\|_{L^2_x}+ \left(|I_j|^{\theta_1} + |I_j|^{\theta_2} \right) \|\scal{\nabla} u\|^\alpha_{S(L^2, I_j)} \|2|s|\nabla v\|_{S(L^2, I_j)} \\
& \mathrel{\phantom{\lesssim \|H(t_j) u(t_j)\|_{L^2_x} }} + \left(|I_j|^{1+\theta_1} + |I_j|^{1+\theta_2} \right) \|\scal{\nabla} u\|^\alpha_{S(L^2, I_j)} \|u\|_{S(L^2, I_j)} \\
&\lesssim \|H(t_j) u(t_j)\|_{L^2_x}+ \left(\eps^{\theta_1} + \eps^{\theta_2} \right) \|u\|^\alpha_{S(I_j)} \|Hu\|_{S(L^2, I_j)} \\
&  \mathrel{\phantom{\lesssim\|H(t_j) u(t_j)\|_{L^2_x} }} + \left(\eps^{1+\theta_1} + \eps^{1+\theta_2} \right) \|u\|^{\alpha+1}_{S(I_j)}. 
\end{align*}
Since $\|u\|_{S(\R)}<\infty$, by choosing $\eps>0$ small enough depending on $T, \|u\|_{S(\R)}$, we get
\[
\|Hu\|_{S(L^2, I_j)} \leq C\|H(t_j)u(t_j)\|_{L^2_x} + C,
\]
for some constant $C>0$ independent of $T$. By induction, we get for each $j$,
\[
\|Hu\|_{S(L^2,I_j)} \leq C\|x u_0\|_{L^2_x} + C.
\] 
Summing over all subintervals $I_j$, we complete the proof. 
\end{proof}

\section*{Acknowledgement}
This work was supported in part by the Labex CEMPI (ANR-11-LABX-0007-01). The author would like to express his deep gratitude to his wife - Uyen Cong for her encouragement and support. He also would like to thank the reviewer for his/her helpful comments and suggestions.

\end{document}